\documentclass[12pt,twoside,reqno]{amsart}
\usepackage{amsfonts}
\usepackage{amssymb}
\usepackage[english]{babel}

\makeatletter
\@namedef{subjclassname@2010}{%
  \textup{2010} Mathematics Subject Classification}
\makeatother

\numberwithin{equation}{section}

\newtheorem{teo}{Theorem}[section]

\newtheorem{lem}[teo]{Lemma}

\theoremstyle{definition}

\newtheorem{rem}[teo]{Remark}

\numberwithin{equation}{section}

\def\a{\alpha}
\def\b{\beta}
\def\l{\lambda }
\def\g{\gamma }

\def\k{\kappa}
\def\R{\mathbb{R}}
\def\NN{\mathbb{N}}
\def\ZZ{\mathbb{Z}}
\def\d{\delta}

\def\s{\sigma}

\begin{document}


\title[On Gagliardo-Nirenberg type inequalities]{On Gagliardo-Nirenberg type inequalities}

\author[V.I. Kolyada and F.J. P\'erez L\'azaro]{V.I. Kolyada and F.J. P\'erez L\'azaro}

\address{Department of Mathematics\\ Karlstad University\\ Universitetsgatan 1
\\651 88 Karlstad\\ Sweden}
\email{viktor.kolyada@kau.se}

\address{Departamento de Matem\'aticas y
Computaci\'on\\Universidad de La Rioja\\ Edif. Vives, c. Luis de
Ulloa\\ 26004 Logro\~no\\Spain}

\email{javier.perezl@unirioja.es}

\thanks{Research of the second named author was partially supported by Grant MTM2012-36732-C03-02 of the D.G.I. of Spain}


\begin{abstract}
We present a Gagliardo-Nirenberg inequality which bounds  Lorentz
norms of the function by Sobolev norms and homogeneous Besov
quasinorms with negative smoothness. We prove also other versions
involving Besov or Triebel-Lizorkin quasinorms. These inequalities can be considered as refinements of Sobolev type
embeddings. They can also be applied to obtain Gagliardo-Nirenberg
inequalities in some limiting cases. Our methods are based on
estimates of rearrangements in terms of heat kernels. These methods
enable us to cover also the case of Sobolev norms with $p=1$.

\end{abstract}


\maketitle

\date{}

\maketitle

\section{Introduction}
In this paper we establish Gagliardo-Nirenberg type inequalities for
Sobolev, Besov and Triebel-Lizorkin spaces.

Recently, some Gagliardo-Nirenberg inequalities have been developed
as a refinement of Sobolev inequalities. Let $f$ be a function on
$\mathbb{R}^n$ such that its distribution function $\l_f(y)$ is
finite. The Gagliardo-Nirenberg-Sobolev  embedding theorem assures
that
\begin{equation}\label{Sobolev_p1}
\|f\|_{n/(n-1)}\le c\|\nabla f\|_1,
\end{equation}
where $c$ only depends on $n$. In the works of Cohen-Meyer-Oru
\cite{CMO}, Cohen-DeVore-Petrushev-Xu \cite{CDPX},
Cohen-Dahmen-Daubechies-DeVore \cite{CDDD} it is proved that
\begin{equation}\label{sobolev_improved}
\|f\|_{n/(n-1)}\le c \|\nabla f\|_1^{\frac{n-1}{n}}\|f\|^{\frac{1}{n}}_{\dot{B}^{-(n-1)}_{\infty,\infty}},
\end{equation}
where $\dot{B}^{-(n-1)}_{\infty,\infty}$ is the homogeneous Besov
space of indices $(-(n-1),\infty,\infty)$. This improved Sobolev
inequality is easily seen to be sharper than (\ref{Sobolev_p1})
(indeed, inequality (\ref{P_h}) below imply $L^{n/(n-1)}\subset \dot{B}^{-(n-1)}_{\infty,\infty}$). Inequality (\ref{sobolev_improved}) presents an
additional feature: it is invariant under the Weil-Heisenberg group
action (see \cite{CMO}). The proof of (\ref{sobolev_improved}) in
\cite{CMO,CDPX,CDDD} is based on wavelet decompositions together
with weak-$\ell^1$ type estimates and interpolation results.

Ledoux \cite{Led} extended inequality (\ref{sobolev_improved}). He
proved that for any $f\in W_p^1(\mathbb{R}^n)$
\begin{equation}\label{ineq_ledoux}
\|f\|_q\le c\|\nabla f\|_p^\theta \|f\|^{1-\theta}_{\dot{B}^{\theta/(\theta-1)}_{\infty,\infty}},\quad 1\le p <q<\infty,\quad \theta= p/q.
\end{equation}
His approach relied on pseudo-Poincar\'e inequalities for heat
kernels.

In particular, inequality (\ref{ineq_ledoux}) gives a refinement of
the Sobolev embedding
\begin{equation}\label{sobolev_con_p}
\|f\|_{np/(n-p)}\le c\|\nabla f\|_p\quad 1\le p<n.
\end{equation}

Afterwards, Mart\'in and Milman \cite{MM} proved an estimate based
on non-increasing rearrangements:
\begin{equation*}
f^{**}(s)\le c|\nabla f|^{**}(s)^{\frac{|\alpha|}{1+|\alpha|}} \|f\|^\frac{1}{1+|\alpha|}_{\dot{B}^\alpha_{\infty,\infty}}, \quad \alpha<0.
\end{equation*}
(here $f^{**}(s)=\frac{1}{s}\int_0^s f^*(t)dt$ and $f^*$ is the
non-increasing rearrangement of $f$). This estimate implies
(\ref{ineq_ledoux}) for $p>1.$ However, since the operator $f\mapsto
f^{**}$ is not bounded in $L^1$, the important case $p=1$ is
unclear.

In this paper we extend inequality (\ref{ineq_ledoux}) to stronger
Lorentz quasinorms  and higher order derivatives . It is well known
that the Sobolev inequality (\ref{sobolev_con_p}) can be improved in
terms of Lorentz spaces. Namely, let $r\in\NN,\,\,1\le r<n, \,\,
1\le p<n/r,$ and let $p^*=np/(n-rp).$ Then for any function $f\in
W_p^r(\R^n)$
\begin{equation}\label{Sobolev_Lorentz}
\|f\|_{p^*,p}\le c\|\mathcal{D}^r f\|_p,
\end{equation}
where
$$
\mathcal{D}^rf(x)=\sum_{|\nu|=r}|D^\nu f(x)|.
$$

We prove that for the same values of parameters,
\begin{equation}\label{refinement}
\|f\|_{p^*,p}\le \|\mathcal{D}^rf\|_p^{1-pr/n}\|f\|_{\dot{B}_{\infty,p}^{r-n/p}}^{pr/n},\quad 1\le p<\frac{n}{r}
\end{equation}
 (see Theorem \ref{Sobolev2} below). This is a refinement of
 (\ref{Sobolev_Lorentz}).

Inequality (\ref{refinement}) is a special case of one of our main
results, Theorem \ref{Sobolev1}. This theorem states the following.
Let $1\le p_1,p_2\le \infty$ and $1\le q_1,q_2\le \infty$. Assume
that $p_1\neq p_2$, $q_1=1$ if $p_1=1$, and $q_i=\infty$  if
$p_i=\infty$ ($i=1,2$). Let $r\in\NN,$ $s<0$ and set
$\theta=r/(r-s)$. Let
\begin{equation*}
\frac{1}{p}=\frac{1-\theta}{p_1}+\frac{\theta}{p_2},\qquad \frac{1}{q}=\frac{1-\theta}{q_1}+\frac{\theta}{q_2}.
\end{equation*}
Then,  for any function $f\in W^r_{p_1,q_1}(\R^n)\cap
\dot{B}^s_{p_2,q_2}(\R^n)$,
\begin{equation}\label{des_main}
 \|f\|_{p,q}\le c
\|\mathcal{D}^rf\|_{p_1,q_1}^{1-\theta}\|f\|_{\dot{B}^s_{p_2,q_2}}^\theta,
\end{equation}
where $c$ doesn't depend on $f.$

It is obvious that (\ref{ineq_ledoux}) can be obtained as a special case
of (\ref{des_main}).

We emphasize that the proof of this result is straightforward and
uses only elementary reasonings. In particular, it doesn't use the
Littlewood-Paley theory.  On the other hand, this theory establishes
the equivalence between Sobolev spaces $W_p^r$ and Lizorkin-Triebel
spaces $F_{p,2}^r$ for $1<p<\infty.$ Therefore for $p_1>1$ Theorem
\ref{Sobolev1} can be partly derived from Gagliardo-Nirenberg
inequalities which we prove for Triebel-Lizorkin and Besov spaces.
We shall briefly describe these results.

First we observe that limiting embeddings into Lorentz spaces
similar to (\ref{Sobolev_Lorentz}) hold also for Besov spaces. Let
$0<r<\infty,\,\,1\le p<n/r, 1\le q\le \infty,$ and let
$p^*=np/(n-rp).$ Then for any function $f$ in the Besov space
$B_{p,q}^r(\R^n)$
\begin{equation}\label{Herz-Peetre}
\|f\|_{p^*,q}\le c\|f\|_{\dot{B}_{p,q}^r}.
\end{equation}
(see \cite{Herz}, \cite{Pe}). In the case $p=q$ a refinement of this
inequality was proved by Bahouri and Cohen \cite{BC}. Namely, they
proved that if $1\le p<n/r$ ($r>0$) and  $p^*=np/(n-rp),$ then
\begin{equation}\label{Bah-Cohen}
\|f\|_{p^*,p}\le c \|f\|_{\dot{B}^r_{p,p}}^{1-pr/n}\|f\|_{\dot{B}_{\infty,p}^{r-n/p}}^{pr/n}.
\end{equation}

In section 6 below  we prove various inequalities similar to
(\ref{des_main}), in which the quasinorms in the right-hand side are
both of Besov type (see Theorem \ref{teoBB}),  or both of
Triebel-Lizorkin-Lorentz type (Theorem \ref{teoFF}),  or represent a
mixture involving Besov and Triebel-Lizorkin types (Theorems \ref{tFB}
and \ref{last}). The exact conditions on the parameters will be
specified in these theorems; here we consider only some special
cases.

An important special case of  Theorem \ref{teoBB} is inequality
(\ref{Bah-Cohen}) and, more generally, a refinement of inequality
(\ref{Herz-Peetre}) for all $1\le q\le \infty,$ that is,
$$
\|f\|_{p^*,q}\le c \|f\|_{\dot{B}^r_{p,q}}^{1-pr/n}\|f\|_{\dot{B}_{\infty,q}^{r-n/p}}^{pr/n}.
$$

Further, Ledoux \cite{Led} observed that inequality
(\ref{ineq_ledoux}) implies some limiting cases of
Gagliardo-Nirenberg inequalities. To be more concrete,
(\ref{ineq_ledoux}) implies
\begin{equation*}
\|f\|_q\le c\|\nabla f\|_p^{p/q}\|f\|_r^{1-p/q}, \quad 1\le p<q<\infty,\,\,\frac{1}{q}=\frac{1}{p}-\frac{r}{qn}.
\end{equation*}

Other examples of limiting cases of Gagliardo-Nirenberg inequalities
were proved by Wadade \cite{Wa}. Similar inequalities to
\cite[Theorem 1.1 and Corollary 1.2]{Wa} can be deduced as
consequences of theorems \ref{last}, \ref{teoFF}, and transitivity of
embeddings (see Remark \ref{remark_wadade}). That is, let
$1<p<q<\infty$, $0<r,\rho <\infty$. Then the following inequalities
hold:
\begin{equation}\label{wadade1}
\|f\|_{q}\le c \|f\|_{\dot{B}^{n/r}_{r,\rho}}^{1-p/q}\|f\|_p^{p/q}
\end{equation}
and
\begin{equation}\label{wadade2}
\|f\|_{q}\le c \|f\|_{\dot{F}^{n/r}_{r,\infty}}^{1-p/q}\|f\|_p^{p/q}.
\end{equation}
Here $\dot{F}^{n/r}_{r,\infty}$ denotes the corresponding
homogeneous Triebel-Lizorkin quasinorm. Let us remark that, in spite
of inequalities (\ref{wadade1}) and (\ref{wadade2}) seem the same as
those in \cite{Wa}, the range of the parameters $p$, $q$, $r$,
$\rho$ where they hold is different. Thus the behaviour of the
constants $c$ is rather different.

The paper is organized as follows. Section 2 contains definitions
and some basic results which are used in the sequel. In Section 3 we
give auxiliary propositions which we apply in Gagliardo-Nirenberg
inequalities involving Sobolev norms. These inequalities are proved
in Section 4. Section 5 contains auxiliary propositions  for
inequalities involving Triebel-Lizorkin and Besov norms. These
inequalities are proved in Section 6.

Our approach is based on estimates of rearrangements in terms of
heat kernels and derivatives.
 We use truncations and corresponding  decompositions (cf. \cite{PW}) to deal with the important case of Sobolev norm in $L^1$. Also, transitivity of embeddings is applied to obtain some results.

\section{Definitions and basic properties}

 Denote by $S_0(\mathbb{R}^n)$ the class of all measurable and
almost everywhere finite functions $f$ on $\mathbb{R}^n$ such that
for each $y>0$
\begin{equation*}
\lambda_f (y) \equiv | \{x \in \mathbb{R}^n : |f(x)|>y \}| <
\infty.
\end{equation*}

A non-increasing rearrangement of a function $f \in
S_0(\mathbb{R}^n)$ is a non-increasing function $f^*$ on
$\mathbb{R}_+ \equiv (0, + \infty)$ such that for any $y>0$
\begin{equation} \label{equi}
|\{t\in \mathbb{R}_+ : f^*(t)>y\}|= \lambda_f (y).
\end{equation}
We shall assume in addition that the rearrangement $f^*$ is left
continuous on $(0,\infty).$ Under this condition it is defined
uniquely by
$$
f^*(t)=\inf\{y>0: \lambda_f (y)<t\}, \quad 0<t<\infty.
$$
For any $t>0$ and any $f, g\in S_0(\mathbb{R}^n)$
\begin{equation*}
(f+g)^*(2t)\le f^*(t)+g^*(t).
\end{equation*}
The following relation holds \cite[Ch. 5]{SW}
\begin{equation} \label{sup}
\sup_{|E|=t} \int_E |f(x)| dx = \int_0^t f^*(u) du \,.
\end{equation}
In what follows we denote
\begin{equation*}
f^{**}(t)= \frac{1}{t} \int_0^t f^*(u) du.
\end{equation*}
By (\ref{sup}), the operator $f\mapsto f^{**}$ is subadditive,
$$
(f+g)^{**}(t)\le f^{**}(t)+g^{**}(t).
$$

 Let $0<p,r<\infty.$ A function $f \in S_0(\mathbb{R}^n)$ belongs
to the Lorentz space $L^{p,r}(\mathbb{R}^n)$ if
\begin{equation*}
\|f\|_{p,r} \equiv \left( \int_0^\infty \left( t^{1/p} f^*(t)
\right)^r \frac{dt}{t} \right)^{1/r} < \infty.
\end{equation*}
For $0<p\le \infty,$ the space $L^{p,\infty}(\mathbb{R}^n)$ is defined
as the class of all $f \in S_0(\mathbb{R}^n)$ such that
$$
\|f\|_{p,\infty} \equiv \sup_{t>0}t^{1/p} f^*(t)<\infty.
$$
We have that $\|f\|_{p,p}=\|f\|_p.$ Further, for a fixed $p$, the
Lorentz spaces $L^{p,r}$ strictly increase as the secondary index
$r$ increases (see \cite[Ch. 4]{BSh}).

We shall use also an alternative expression of Lorentz quasinorms
\begin{equation}\label{lor_alt}
\|f\|_{p,r}=\left(p\int_0^\infty
y^{r-1}\l_f(y)^{r/p}\,dy\right)^{1/r}, \quad 0<p,r<\infty
\end{equation}
(see \cite[Proposition 1.4.9]{Graf}).

Let $1\le p\le \infty$ and $r\in\NN.$ Denote by $W_p^r(\R^n)$ the
Sobolev space of functions $f\in L^p(\R^n)$ for which all weak
derivatives $D^\nu f$ $(\nu=(\nu_1,...,\nu_n))$  of order
$|\nu|=\nu_1+\cdots +\nu_n\le r$ exist and belong to $L^p(\R^n).$

Further, we shall consider the homogeneous Besov spaces and the
homogeneous Triebel-Lizorkin spaces. These spaces have a wide
history.
 They admit several equivalent definitions in terms of moduli of smoothness, approximations, Littlewood-Paley decompositions,
 Cauchy-Poisson semigroup, Gauss-Weierstrass semigroup,  wavelet decompositions (see \cite{Pee, Tri, Tri2, Tri3}).
 In this paper we deal with the thermic description based on the Gauss-Weierstrass semigroup.

 From now on, define for any $x,y\in \R^n$,
\begin{equation*}
p_h(y)=\frac{e^{-|y|^2/(4h)}}{(4\pi h)^{n/2}},\qquad
P_hf(x)= \int_{\mathbb{R}^n} p_h(y)f(x-y)dy.
\end{equation*}

By H\"older's inequality, for any $f\in L^q(\R^n),\,\,1\le q\le
\infty,$
\begin{equation}\label{P_h}
\|P_hf\|_\infty\le ch^{-n/(2q)}\|f\|_q.
\end{equation}

 Let $-\infty<s<\infty$, $0<q\le \infty,$ and $0<p\le \infty$.  Let $m$ be a non-negative integer such that $2m>s$. The homogeneous Besov space
 $\dot{B}^s_{p,q}(\R^n)$ is defined as the space of all tempered distributions $f\in S'$ on $\R^n$ such that
\begin{equation*}
\|f\|_{\dot{B}^s_{p,q}}=\left(\int_0^\infty h^{(m-s/2)q}\left\|\frac{\partial^m P_hf}{\partial h^m}\right\|_p^q\frac{dh}{h}\right)^{1/q}<\infty
\end{equation*}
(usual modification if $q=\infty$).

It is well known that Besov spaces $\dot{B}^s_{p,q}$ increase as the
second index $q$ increases, that is
\begin{equation}\label{second_index}
\|f\|_{\dot{B}^s_{p,q}}\le c\|f\|_{\dot{B}^s_{p,r}},\quad 0<r<q\le \infty.
\end{equation}

Furthermore,  if $0<p_0<p_1\le \infty,\,\, 0<q\le \infty, \,
-\infty<s_0<\infty,$ and $s_1=s_0-n(1/p_0-1/p_1),$ then
\begin{equation}\label{embed1}
\|f\|_{\dot{B}^{s_1}_{p_1,q}}\le c \|f\|_{\dot{B}^{s_0}_{p_0,q}}
\end{equation}
(see \cite[2.7.1]{Tri}).

 We have also
the following inequality: if $1\le p_0<p_1\le\infty$, $n\ge 2$ if $p_0=1$, $r\in\NN,$ and
$s=r-n(1/p_0-1/p_1),$ then for any function $f\in W_{p_0}^r(\R^n)$
\begin{equation}\label{embed2}
\|f\|_{\dot{B}^s_{p_1,p_0}}\le c \sum_{|\nu|=r}\|D^\nu f\|_{p_0}.
\end{equation}
By (\ref{embed1}), it is sufficient to obtain (\ref{embed2}) in the
case $s>0$;  for this case, see   \cite{Herz}, \cite{Ko88} --
\cite{KP},  and references therein.

We recall also the thermic definition of Triebel-Lizorkin spaces.
Let $-\infty<s<\infty$, $0<p< \infty$ and $0<q\le \infty$.  Let
$m$ be a non-negative integer such that $2m>s$. The homogeneous
Triebel-Lizorkin space $\dot{F}^s_{p,q}(\R^n)$ is defined as the
space of all tempered distributions $f\in S'$ on $\R^n$ such that
\begin{equation*}
\|f\|_{\dot{F}^s_{p,q}}=\left\|\left(\int_0^\infty h^{(m-s/2)q}\left|\frac{\partial^m P_hf}{\partial h^m}(\cdot)\right|^q\frac{dh}{h}\right)^{1/q}\right\|_p
\end{equation*}
(usual modification if $q=\infty$).

For fixed $p$ the Tribel-Lizorkin spaces $\dot{F}^s_{p,q}(\R^n)$
increase as the  index $q$ increases.

In order to obtain more precise Gagliardo-Nirenberg inequalities
we will consider also the  Triebel-Lizorkin  spaces based on Lorentz
quasinorms. Let $-\infty<s<\infty$, $0<p<\infty$, $0<q,r \le
\infty.$ We say that a tempered distribution $f$ belongs to
$\dot{F}^s_{p,r;q}(\R^n)$ if
 \begin{equation*}
\left\|\left(\int_0^\infty h^{(m-s/2)q}\left|\frac{\partial^m P_hf}{\partial h^m}(\cdot)\right|^q\frac{dh}{h}\right)^{1/q}\right\|_{p,r}<\infty
\end{equation*}
(usual modification if $q=\infty$). Observe that quasinorms of this
kind were considered in \cite{MM}. The corresponding quasinorms
based on Littlewood-Paley decompositions were also used in \cite{XY,
YCP, ZWT}.

\section{Auxiliary propositions for inequalities involving  Sobolev norms}

The following lemma is a slight modification of Lemma 2.4 in
\cite{KP}.
\begin{lem}\label{KP}
Let $\{\alpha_k\}_{k\in \mathbb{Z}}\in\ell^1$ be a nonzero sequence
of nonnegative numbers, and let $0<\delta<\infty$. Then there exists
a sequence $\{\beta_k\}_{k\in\mathbb{Z}}$ of positive numbers
satisfying the following conditions:
\begin{enumerate}

\item $\displaystyle \alpha_k\le \beta_k$ for all $k\in \mathbb{Z}$;

\item $\displaystyle \sum_{k\in\ZZ}\beta_k = \frac{1}{(1-2^{-\delta})^{2}}\sum_{k\in\mathbb{Z}}\alpha_k$;

\item $\displaystyle 2^{-\delta}\le \beta_{k+1}/\beta_k\le 2^\delta$, $k\in\ZZ$.
\end{enumerate}

\end{lem}

\begin{proof}
Define
\begin{equation*}
\alpha'_k=2^{-k\delta}\sum_{m\le k} 2^{m\delta} \alpha_m,\quad k\in\mathbb{Z}.
\end{equation*}
Then $\alpha'_k\ge \alpha_k$ and
\begin{equation}\label{sumaalphaprima}
\sum_{k\in\ZZ}\alpha'_k =\sum_{m\in \ZZ}2^{m\delta}\alpha_m \sum_{k=m}^\infty 2^{-k\delta}=\frac{1}{1-2^{-\delta}}\sum_{m\in\ZZ}\alpha_m.
\end{equation}
Since $2^{k\delta}\alpha'_k$ increases, we have
\begin{equation}\label{cocientealphaprima}
\alpha'_{k+1}\ge 2^{-\delta}\alpha'_k,\quad k\in\ZZ.
\end{equation}
Further, set
\begin{equation*}
\beta_k=2^{k\delta}\sum_{m=k}^\infty 2^{-m\delta}\alpha'_m.
\end{equation*}
Then $\beta_k\ge \alpha'_k\ge \alpha_k$. By (\ref{sumaalphaprima}),
we have also
\begin{equation*}
\sum_{k\in\ZZ}\beta_k= \sum_{k\in\ZZ}2^{k\delta}\sum_{m=k}^\infty 2^{-m\delta}\alpha'_m=\sum_{m\in \ZZ}2^{-m\delta}\alpha'_m\sum_{k=-\infty}^m 2^{k\delta}=
\end{equation*}
\begin{equation*}
=\frac{1}{1-2^{-\delta}}\sum_{m\in \ZZ}\alpha'_m=\frac{1}{(1-2^{-\delta})^2}\sum_{m\in \ZZ}\alpha_m.
\end{equation*}
Since $\{\alpha_k\}_{k\in\ZZ}$ is nonzero, we have $\alpha'_k>0$ for
$k\ge k_0$. Thus, $\beta_k>0$ ($k\in\ZZ$). Since
$2^{-k\delta}\beta_k$ decreases, we have
\begin{equation*}
\beta_k\ge 2^{-\delta}\beta_{k+1}.
\end{equation*}
On the other hand, using (\ref{cocientealphaprima}), we obtain
$$
\beta_k= 2^{k\delta}\sum_{m=k}^\infty 2^{-m\delta}\alpha'_m\le  2^{(k+1)\delta}\sum_{m=k}^\infty 2^{-m\delta}\alpha'_{m+1}
$$
$$
=2^{(k+1)\delta}\sum_{m=k+1}^\infty
2^{-(m-1)\delta}\alpha'_{m}=2^\delta \beta_{k+1}.
$$
\end{proof}

\begin{lem}\label{sets}
Let $J\subset \mathbb{Z}$ and $\{E_j\}_{j\in J}\subset \mathbb{R}^n$ be a sequence of measurable disjoint sets such that for any $j\in J$,
\begin{equation*}
\mu_j=\sum_{ k\in J, k\ge j}|E_k|>0
\end{equation*}
Let $1\le q\le p<\infty$. Then for any function $f\in L^{p,q}(\mathbb{R}^n)$
\begin{equation}\label{lemasetsmain}
\sum_{j\in J} \mu_j^{q/p-1}\int_{E_j}|f(x)|^qdx \le \|f\|_{q,p}^q.
\end{equation}
\end{lem}
\begin{proof}
Observe that
\begin{equation*}
\sum_{j\in J} \mu_j^{q/p-1}\int_{E_j}|f(x)|^qdx = \int_{\mathbb{R}^n}G(x)|f(x)|^q dx,
\end{equation*}
where 
\begin{equation*}
G(x)=\sum_{j\in J} \mu_j^{q/p-1}\chi_{E_j}(x).
\end{equation*}
Since $q\le p$, it holds that $\mu_j^{q/p-1}$ increases as $j$ increases. Then, if $y\in E_j$,
\begin{equation*}
G(y)=\mu_j^{q/p-1}\le G(x)\text{ if }x\in \bigcup_{k\in J, k>j}E_k
\end{equation*}
and
\begin{equation*}
G(y)\ge G(x) \text{ if }x\in \bigcup_{k\in J, k<j}E_k.
\end{equation*}  In consequence,
\begin{equation*}
G^*(u)=\sum_{j\in J} \mu_j^{q/p-1}\chi_{(\mu_{j+1},\mu_j]}(u). 
\end{equation*}
Further, for $u\in (\mu_{j+1},\mu_j]$ it holds that $\mu_j^{q/p-1}\le u^{q/p-1}$, hence $G^*(u)\le u^{q/p-1}$. Applying Hardy-Littlewood inequality, we obtain
\begin{equation*}
\int_{\mathbb{R}^n}G(x)|f(x)|^q dx\le \int_0^\infty G^*(u)f^*(u)^qdu\le \int_0^\infty u^{q/p-1}f^*(u)^qdu.
\end{equation*}
This implies (\ref{lemasetsmain}).
\end{proof}

\begin{lem}\label{functions}
Let $1< p\le q<\infty.$ Assume that $f\in L^{p,q}(\R^n)$ and let
$\{E_j\}_{j\in \ZZ}$ be a sequence of  measurable sets such that for
some $N\in\NN$
$$
E_j\cap E_k=\emptyset \quad\mbox{if}\quad |j-k|\ge N.
$$
Then
\begin{equation}\label{lorentz}
\sum_{j\in\ZZ} \|f\chi_{E_j}\|_{p,q}^q\le  N^{q/p} \|f\|_{p,q}^q.
\end{equation}
\end{lem}
\begin{proof}
Note that $E_{k+jN}\cap E_{k+iN}=\emptyset$ for any $k,j, i\in\mathbb{Z}$, $i\neq j$. Denote $f_j=f\chi_{E_j}$. We have for any $0\le k<N$, 
\begin{equation*}
\sum_{j\in\mathbb{Z}}\lambda_{f_{k+jN}}(y)=\sum_{j\in\mathbb{Z}} |\{x\in E_{k+jN}: |f(x)|>y\}|\le \lambda_f(y).
\end{equation*}
Then
\begin{equation*}
\sum_{j\in\mathbb{Z}}\lambda_{f_j}(y)=\sum_{k=0}^{N-1} \sum_{j\in\mathbb{Z}}\lambda_{f_{k+jN}}(y)\le N\lambda_f(y) \,\text{ for any }\,y>0. 
\end{equation*}

Thus, using (\ref{lor_alt}) and taking into account that $p\le q$ we get,
$$
\begin{aligned}
\sum_{j\in\ZZ} \|f_j\|_{p,q}^q&=p \sum_{j\in\ZZ} \int_0^\infty y^{q-1}\left(\l_{f_j}(y)\right)^{q/p}\,dy\\
&\le p\int_0^\infty y^{q-1}\left(\sum_{j\in\ZZ} \l_{f_j}(y)\right)^{q/p}\,dy\\
&\le N^{q/p} p\int_0^\infty y^{q-1}\l_f(y)^{q/p}\,dy= N^{q/p}\|f\|_{p,q}^q.
\end{aligned}
$$
\end{proof}

As above, we denote
$$
p_h(y)=(4\pi h)^{-n/2}e^{-|y|^2/(4h)}.
$$

\begin{lem}\label{heat} Assume that a function  $f\in L_{loc}^1(\R^n)$   has a weak gradient $\nabla f\in  S_0(\R^n)$ such that
\begin{equation}\label{gradient}
\int_0^t(\nabla f)^*(s)\,ds<\infty \quad\mbox{for any}\quad t>0.
\end{equation}
Then for any $h>0$
\begin{equation}\label{Pf}
\int_{\R^n}p_h(x-y)|f(y)|\,dy<\infty \quad\mbox{for almost all}\quad x\in \R^n
\end{equation}
and
\begin{equation}\label{f-Pf}
(f-P_hf)^{**}(t)\le c_n\sqrt{h}(\nabla f)^{**}(t)  \quad\mbox{for any}\quad t>0,
\end{equation}
  where  $c_n$ depends only on $n$.
\end{lem}
\begin{proof} For almost every $x\in \R^n$ and almost every $v\in \R^n$ we have
\begin{equation}\label{diff}
f(x+v)-f(x)=\int_0^1\nabla f(x+\tau v)\cdot v\, d\tau
\end{equation}
(see \cite[p. 143]{LL}). Thus,
$$
|f(x+v)|\le |f(x)|+|v|\int_0^1|\nabla f(x+\tau v)|\, d\tau.
$$
From here, we obtain that for any cube $Q\subset \R^n$
$$
\int_Q \int_{\R^n}p_h(v)|f(x+v)|\,dv\,dx\le \int_Q|f(x)|\,dx
$$
$$
+ \int_0^1\int_{\R^n}p_h(v)|v|\int_Q|\nabla f(x+\tau v)|\,dx\,\,dv\, d\tau
$$
$$
\le \int_Q|f(x)|\,dx+ c_n\sqrt{h}\int_0^{|Q|}(\nabla f)^*(s)\,ds<\infty,
$$
where
$$
c_n=2\pi^{-n/2}\int_{\R^n}|v|e^{-|v|^2}\,dv.
$$
This implies (\ref{Pf}).

Further, for any $h>0$ we have
$$
f(x)-P_hf(x)= (4\pi h)^{-n/2}\int_{\R^n}e^{-|x-y|^2/(4h)}[f(x)-f(y)]dy
$$
$$
=\pi^{-n/2}\int_{\R^n}e^{-|z|^2}[f(x)-f(x+2\sqrt{h}z)]dz.
$$
Using (\ref{diff}), we get
$$
|f(x)-P_hf(x)|\le 2\sqrt{h}\pi^{-n/2}\int_{\R^n}|z|e^{-|z|^2}\int_0^1|\nabla f(x+2\sqrt{h}\tau z)|\,d\tau dz.
$$
By this estimate, we have for any measurable set $E\subset \R^n$
with measure $|E|=t>0$
$$
\begin{aligned}
&\int_E |f(x)-P_hf(x)|dx\\
&\le 2\sqrt{h}\pi^{-n/2}\int_{\R^n}|z|e^{-|z|^2}\int_0^1\int_E|\nabla f(x+2\sqrt{h}\tau z)|\,dx\,d\tau dz\\
&\le 2\sqrt{h}\pi^{-n/2}\int_0^t (\nabla f)^{*}(u)du \int_{\R^n}|z|e^{-|z|^2}dz.
\end{aligned}
$$
This implies (\ref{f-Pf}).
\end{proof}

\begin{rem}\label{M-M}
We observe that inequality (\ref{f-Pf}) was proved in \cite{MM}
(with the use of $K-$functionals). We give a direct proof of this
inequality for completeness.

\end{rem}

\begin{rem}\label{P**}
Assume that $f\in S_0(\R^n)$ and
\begin{equation}\label{f**}
f^{**}(t)<\infty \quad\mbox{for any}\quad t>0.
\end{equation}
Then
\begin{equation}\label{Pf**}
(P_hf)^{**}(t)\le f^{**}(t)\quad\mbox{for all}\quad t>0.
\end{equation}

Indeed, for any measurable set $E\subset \R^n$ with measure $|E|=t$
we have
$$
\int_E|P_hf(x)|\,dx\le \int_{\R^n}p_h(z)\int_E|f(x-z)|dxdz\le tf^{**}(t).
$$
This implies (\ref{Pf**}).

It is possible to prove that (\ref{f**}) holds for any function $f$
satisfying conditions of Lemma \ref{heat}.
\end{rem}

\begin{lem}\label{lifting_termico}
Let $\alpha=(\alpha_1,\ldots,\alpha_n)\in \mathbb{Z}^n$ be a non-negative multi-index and set $|\alpha|=\alpha_1+\ldots+\alpha_n$. Assume that $f$ is a locally integrable function 
, which has weak derivative $D^\alpha f$. Assume also that $f$ is a tempered distribution. Let $s<0$, $1\le p,q\le \infty$, then
\begin{equation*}
\|D^\alpha f\|_{\dot{B}_{p,q}^{s-|\alpha|}(\mathbb{R}^n)}\le c \|f\|_{\dot{B}_{p,q}^{s}(\mathbb{R}^n)},
\end{equation*}
where $c$ only depends on $n$, $\alpha$ and $s$.
\end{lem}
This lemma is well known. See, for instance \cite[p.59, 242]{Tri}. But there, the norms in homogeneous Besov spaces are taken in terms of Littlewood-Paley decompositions. For completeness, we present a proof using the thermic description of the Besov norm. 

\begin{proof}
First, as $f$ is a tempered distribution and $p_h$ is in the Schwartz class, it is well known (cf. \cite[p.52-53]{Stri})that the convolution $P_hf=f*p_h$ is a $C^\infty(\mathbb{R}^n)$ function which is also a tempered distribution and it holds that
\begin{equation}\label{convolutionSSprima}
D^\alpha (P_h f)=p_h*(D^\alpha f)= f*(D^\alpha p_h).
\end{equation}
It is also known that \cite[p.393, Theorem 2 (ii)]{Flett} for any $h>0$ and $g$ in $L^p(\mathbb{R}^n)$
\begin{equation}\label{des_Flett}
\|D^\alpha(P_h g)\|_p\le c_{n,\alpha} h^{-|\alpha|/2}\|g\|_p
\end{equation}
Moreover, since $p_{2h}=p_h*p_h$, we have
\begin{equation}\label{pdosh}
P_{2h}f=P_h(P_h f).
\end{equation}
Then, by (\ref{convolutionSSprima}), (\ref{pdosh}) and (\ref{des_Flett}) we obtain
\begin{equation*}
\|P_{2h} (D^\alpha f)\|_{p}=\|D^\alpha(P_{2h}f)\|_p=\|D^\alpha(P_{h}(P_{h}f))\|_p\le c_{n,\alpha} h^{-|\alpha|/2}\|P_{h}f\|_p 
\end{equation*}
From this inequality, the Lemma immediately follows. 
\end{proof}

\section{Inequalities with Sobolev norms}

By $W_{p,q}^r(\R^n)$ ($r\in\NN$) we denote the space of all
functions $f\in L^{p,q}(\R^n)$ for which all weak derivatives $D^\nu
f$ $(\nu=(\nu_1,...,\nu_n))$  of order $|\nu|=\nu_1+\cdots +\nu_n\le
r$ exist and belong to $L^{p,q}(\R^n).$ As above, we denote
$$
\mathcal{D}^rf(x)=\sum_{|\nu|=r}|D^\nu f(x)|.
$$

\begin{teo}\label{Sobolev1}
Let $1\le p_1,p_2\le \infty$ and $1\le q_1,q_2\le \infty$. Assume
that $p_1\neq p_2$,  $q_1=1$ if $p_1=1$, and $q_i=\infty$  if
$p_i=\infty$  ($i=1,2$). Let $r\in\NN,$ $s<0$ and set
$\theta=r/(r-s)$. Let
\begin{equation}\label{100}
\frac{1}{p}=\frac{1-\theta}{p_1}+\frac{\theta}{p_2},\qquad \frac{1}{q}=\frac{1-\theta}{q_1}+\frac{\theta}{q_2}.
\end{equation}
Then, for any function $f\in W^r_{p_1,q_1}(\R^n)\cap
\dot{B}^s_{p_2,q_2}(\R^n)$,
\begin{equation}\label{1}
\|f\|_{p,q}\le c \|\mathcal{D}^rf\|_{p_1,q_1}^{1-\theta}\|f\|_{\dot{B}^s_{p_2,q_2}}^\theta,
\end{equation}
where $c$ doesn't depend on $f.$
\end{teo}

\begin{proof} First we consider the case $r=1$.

For any $A\ge 0,$ set
$$
F(x,A)= \min(|f(x)|, A)\operatorname{sign}(f(x)).
$$
The same  reasonings as in \cite[2.1.4, 2.1.8]{Zi} show that $F(x,A)$ can be modified on a set of measure zero so that the modified function is locally absolutely continuous on almost all lines parallel to the coordinate axes (we shall call this  {\it the $W-$property}). Set now for any $j\in \ZZ$
$$
f_j(x)=F(x, f^*(2^{-j-\nu}))- F(x, f^*(2^{-j+\nu}))
$$
(where a number $\nu\in\NN$ will be chosen later). Then each $f_j$ has the $W-$property. 
Let
\begin{equation*}
H_j=\{x\in\R^n: f^*(2^{-j+\nu})<|f(x)|<f^*(2^{-j-\nu})\},
\end{equation*}
(clearly some $H_j$ may be empty). Then $\nabla f_j(x)=0$ for almost all $x\notin H_j$ and $\nabla
f_j(x)=\nabla f(x)$ for almost all $x\in H_j$. Thus,
\begin{equation}\label{2}
\nabla f_j(x)=\chi_{H_j}(x)\nabla f(x)\qquad \text{for almost all}\quad x\in\R^n.
\end{equation}
It follows from the definition of $H_j$ that
\begin{equation}\label{nonintersect}
H_{j}\cap H_{k}=\emptyset\quad\mbox{if}\quad |j-k|\ge 2\nu.
\end{equation}
Besides, we have
\begin{equation}\label{measures}
|H_j|\le  2^{-j+\nu}  \quad (j\in\NN).
\end{equation}
Denote
\begin{equation}\label{300}
A_j= 2^{j/p_1'}\int_{H_j}|\nabla f(x)|\,dx.
\end{equation}

 We shall show that
\begin{equation}\label{3}
\left(\sum_{j\in \ZZ} A_j^{q_1}\right)^{1/q_1}\le K \|\nabla f\|_{p_1,q_1}, \quad\mbox{where}\,\, K=2^{2\nu}\max\left(1,\frac{p_1'}{q_1'}\right).
\end{equation}

First we assume that $1\le q_1\le p_1<\infty.$ By H\"older's
inequality and (\ref{measures}),
\begin{equation}\label{tres_bis}
A_j^{q_1}\le 2^{jq_1/p_1'}|H_j|^{q_1-1}\int_{H_j}|\nabla f(x)|^{q_1}\,dx.
\end{equation}

For a fixed integer $0\le m<\ 2\nu$, consider the following proper subset of $\mathbb{Z}$:
\begin{equation*}
J=\{2\nu i +m \in \mathbb{Z}: i\in \mathbb{Z}, |H_{2\nu i +m}|>0\}.
\end{equation*}
By (\ref{nonintersect}), the sets $H_j$, $j\in J$ are pairwise disjoint. Further, set
$\mu_j=\sum_{k\in J, k\ge j}|H_k|$ for any $j\in J$.
Note that
\begin{equation}\label{otra_medida}
 0<|H_j|\le \mu_j\le 2^{-j+\nu}
\end{equation}
since 
\begin{equation*}
\bigcup_{k\in J, k\ge j} H_k \subset\{x\in \mathbb{R}^n: f^*(2^{-j+\nu})<|f(x)|\}\quad (j\in J).
\end{equation*}
Then, by (\ref{tres_bis}), (\ref{otra_medida}) and Lemma \ref{sets}, we have
\begin{equation*}
\sum_{i\in \mathbb{Z}}A_{2\nu i +m}^{q_1}= \sum_{j\in J} A_j^{q_1}\le
\end{equation*}
\begin{equation*}
\le 2^{\nu q_1/p_1'}\sum_{j\in J}\mu_j^{q_1/p_1-1}\int_{H_j}|\nabla f(x)|^{q_1}dx \le 2^{\nu q_1/p_1'} \|\nabla f\|_{p_1,q_1}^{q_1}.
\end{equation*}
Thus,
\begin{equation}\label{3-1}
\sum_{j\in \ZZ} A_j^{q_1}\le 2\nu 2^{\nu q_1/p_1'} \|\nabla f\|_{p_1,q_1}^{q_1}\quad \mbox{if}\quad 1\le q_1\le p_1<\infty.
\end{equation}

Let now $1<p_1<q_1<\infty.$ First, we have, applying H\"older's
inequality and taking into account (\ref{2}) and (\ref{measures})
$$
\begin{aligned}
&A_j\le 2^{j/p_1'}\int_0^{|H_j|}(\nabla f_j)^*(t)\,dt
\le 2^{j/p_1'}\left(\int_0^{|H_j|}t^{q_1'/p_1'}\frac{dt}{t}\right)^{1/q_1'}\|\nabla f_j\|_{p_1,q_1}\\
&= 2^{j/p_1'}|H_j|^{1/p_1'}\left(\frac{p_1'}{q_1'}\right)^{1/q_1'}\|\nabla f_j\|_{p_1,q_1}\le
2^{\nu/p_1'}\left(\frac{p_1'}{q_1'}\right)^{1/q_1'}\|\chi_{H_j}\nabla f\|_{p_1,q_1}.
\end{aligned}
$$
Using this estimate, (\ref{nonintersect}) and applying Lemma \ref{functions}, we obtain
that
$$
\sum_{j\in \ZZ} A_j^{q_1}\le (2\nu)^{q_1/p_1} 2^{\nu q_1/p_1'}\left(\frac{p_1'}{q_1'}\right)^{q_1-1}\|\nabla f\|_{p_1,q_1}^{q_1}, \quad 1<p_1<q_1<\infty.
$$
Together with (\ref{3-1}), this implies (\ref{3}) for the case
$p_1<\infty, q_1<\infty.$ In the case $q_1=\infty, 1<p_1\le \infty$
inequality (\ref{3}) is obvious.

We shall estimate $f^{**}(2^{-j})$. Observe that if 
$$f^*(2^{-j})\le |f(x)|\le f^*(2^{-j-\nu}),$$
then 
$$
|f(x)|=|f_j(x)|+f^*(2^{-j+\nu}).
$$
Thus,
$$
f^*(t)=f_j^*(t)+ f^*(2^{-j+\nu}) \quad\mbox{for}\quad 2^{-j-\nu}\le t\le 2^{-j}.
$$ 
Using this observation, we get
$$
f^{**}(2^{-j})= 2^j\left(\int_0^{2^{-j-\nu}}f^*(t)dt+\int_{2^{-j-\nu}}^{2^{-j}}[f_j^*(t)+ f^*(2^{-j+\nu})]dt\right)
$$
\begin{equation}\label{4}
\le 2^{-\nu}f^{**}(2^{-j-\nu}) + f_j^{**}(2^{-j})+f^*(2^{-j+\nu}).
\end{equation}

Further, for any $k\in\ZZ$, choose $h_k\in [2^{-2(k+1)},2^{-2k}]$
such that\footnote{In fact, it is known that $\|P_h f\|_{p_2}$ decreases in $h$ (cf. \cite[Theorem 4(ii)]{Flett}).}
\begin{equation*}
\|P_{h_k}f\|_{p_2}=\min\{ \|P_h f\|_{p_2}: h\in [2^{-2(k+1)},2^{-2k}]\}.
\end{equation*}
We have
\begin{equation}\label{5}
f_j^{**}(2^{-j})\le (f_j-P_{h_k}f_j)^{**}(2^{-j})+ (P_{h_k}f_j)^{**}(2^{-j})
\end{equation}
for all $j,k\in \ZZ.$ By (\ref{f-Pf}),
$$
(f_j-P_{h_k}f_j)^{**}(2^{-j})\le c \sqrt{h_k} (\nabla f_j)^{**}(2^{-j})
$$
\begin{equation}\label{6}
\le c 2^{j-k}\int_{H_j}|\nabla f(x)|\,dx=c2^{j/p_1-k}A_j,
\end{equation}
where $c$ depends only on $n$. Further,
\begin{equation*}
P_{h_k}f_j =P_{h_k}(f_j-f)+P_{h_k}f.
\end{equation*}
We have
\begin{equation}\label{7}
(P_{h_k}f)^{**}(2^{-j})\le 2^{j/p_2}\|P_{h_k}f\|_{p_2}\equiv 2^{j/p_2-ks}\alpha_k,
\end{equation}
where $\alpha_k=2^{ks}\|P_{h_k}f\|_{p_2}$. Besides, by (\ref{Pf**}),
\begin{equation}\label{8}
(P_{h_k}(f_j-f))^{**}(2^{-j})\le (f_j-f)^{**}(2^{-j}).
\end{equation}
If $|f(x)|>f^*(2^{-j-\nu})$, then
$$
|f(x)-f_j(x)|=|f(x)|-f^*(2^{-j-\nu})+f^*(2^{-j+\nu}).
$$
If $|f(x)|\le f^*(2^{-j-\nu})$, then
$$
|f(x)-f_j(x)|\le f^*(2^{-j+\nu}).
$$
Thus,
\begin{equation}\label{9}
(f-f_j)^{**}(2^{-j})\le 2^{-\nu}f^{**}(2^{-j-\nu})+f^*(2^{-j+\nu}).
\end{equation}
Applying inequalities (\ref{4}) -- (\ref{9}), we obtain
$$
f^{**}(2^{-j})\le c [2^{j/p_1-k}A_j+2^{j/p_2-ks}\a_k]
$$
\begin{equation}\label{10}
+2^{-\nu+1}f^{**}(2^{-j-\nu})+2f^*(2^{-j+\nu}),
\end{equation}
where $A_j$ is defined by (\ref{300}) and
$\a_k=2^{ks}\|P_{h_k}f\|_{p_2}.$ Set $d=1/p_1-1/p_2;$ by our
assumption, $d\not=0.$ Choose
\begin{equation}\label{delta}
0<\delta< \min(q_1|d|, \,q_2(1-s)).
\end{equation}
Applying Lemma \ref{KP}, we obtain that there exists a sequence
$\{B_j\}_{j\in\ZZ}$ of positive numbers such that
\begin{equation}\label{11}
A_j\le B_j \text{ for all }j\in\ZZ,
\end{equation}
\begin{equation}\label{12}
\|\{B_j\}\|_{l^{q_1}}\le (1-2^{-\delta})^{-2/q_1}\|\{A_j\}\|_{l^{q_1}},
\end{equation}
and
\begin{equation}\label{13}
2^{-\delta/q_1}\le B_{j+1}/B_j \le 2^{\delta/q_1},\quad j\in\ZZ.
\end{equation}
Further,
\begin{equation*}
\|\{\a_k\}\|_{l^{q_2}}= \left(\sum_{k\in\ZZ}2^{ksq_2}\|P_{h_k}f\|_{p_2}^{q_2}\right)^{1/q_2}\le c\|f\|_{\dot{B}^s_{p_2,q_2}}.
\end{equation*}
Applying  Lemma \ref{KP}, we obtain a sequence
$\{\beta_k\}_{k\in\ZZ}$ of positive numbers such that
\begin{equation}\label{14}
\alpha_k\le \beta_k \text{ for all }k\in\ZZ,
\end{equation}
\begin{equation}\label{15}
\|\{\beta_k\}\|_{\ell^{q_2}}\le c(1-2^{-\d})^{-2/q_2}\|f\|_{\dot{B}^s_{p_2,q_2}},
\end{equation}
and
\begin{equation}\label{16}
2^{-\delta/q_2}\le \beta_{k+1}/\beta_k \le 2^{\delta/q_2}\qquad (k\in\ZZ).
\end{equation}
Now we have from (\ref{10}), (\ref{11}) and (\ref{14})
$$
f^{**}(2^{-j})\le c [2^{j/p_1-k}B_j+2^{j/p_2-ks}\beta_k]
$$
\begin{equation}\label{17}
+2^{-\nu+1}f^{**}(2^{-j-\nu})+2f^*(2^{-j+\nu}).
\end{equation}
Note that, by (\ref{delta}) and (\ref{16}), $2^{k(1-s)}\beta_k$
strictly increases on $k$, and
$$
\lim_{k\to +\infty}2^{k(1-s)}\beta_k=\infty.
$$
Since $\{\b_k\}$ is bounded, we have also that
$$
\lim_{k\to -\infty}2^{k(1-s)}\beta_k=0.
$$
 Thus, for any fixed $j\in\ZZ$
there exists an integer  $\kappa(j)$  such that
\begin{equation}\label{k(j)}
2^{\kappa(j)(1-s)}\beta_{\kappa(j)}\le 2^{jd}B_j< 2^{(\kappa(j)+1)(1-s)}\beta_{\kappa(j)+1}
\end{equation}
(where $d=1/p_1-1/p_2).$ Choose a natural number
$$
N>\frac{1-s+\delta/q_2}{|d|-\delta/q_1}.
$$

Suppose first that $p_1<p_2$ and thus $d>0$. Applying inequalities
(\ref{13}), (\ref{16}), (\ref{k(j)}), and taking into account the
choice of $N$ and $\delta$, we obtain that for any $j\in\ZZ$
$$
\begin{aligned}
&2^{\kappa(j)(1-s)}\beta_{\k(j)}\le 2^{jd}B_j
\le 2^{(j+N)d}B_{j+N} 2^{-N(d-\delta/q_1)}\\
&<2^{(j+N)d}B_{j+N} 2^{-(1-s+\delta/q_2)}
< 2^{\k(j+N)(1-s)}\beta_{\k(j+N)+1}2^{-\delta/q_2}\\
 &\le  2^{\k(j+N)(1-s)}\beta_{\k(j+N)}.
\end{aligned}
$$
Since $2^{k(1-s)}\beta_k$ increases, this inequality implies that
$\k(j)<\k(j+N)$.
 Thus,  using (\ref{15}), we have that
\begin{equation}\label{18}
\left(\sum_{j\in\ZZ}\beta_{\k(j)}^{q_2}\right)^{1/q_2}\le c'N^{1/q_2}\|f\|_{\dot{B}^s_{p_2,q_2}}.
\end{equation}
Now we consider the case $p_2<p_1$ (that is, $d<0$). Following the
same reasonings
 as in the previous case, we get
$$
\begin{aligned}
&2^{\k(j+N)(1-s)}\beta_{\k(j+N)}\le 2^{(j+N)d}B_{j+N}\\
&\le 2^{jd}B_{j} 2^{-N(|d|-\delta/q_1)}
<2^{jd}B_{j}  2^{-(1-s+\delta/q_2)}\\
&< 2^{\k(j)(1-s)}\beta_{\k(j)+1}2^{-\delta/q_2}
 \le 2^{\k(j)(1-s)}\beta_{\k(j)}.
\end{aligned}
$$
Then, $\k(j+N)<\k(j)$, and (\ref{18}) holds in this case, too.

Using inequalities (\ref{17}) and (\ref{k(j)}), and taking into
account that
$$
\theta(1-s)=1\quad\mbox{and}\quad \frac{1}{p}=\frac{1-\theta}{p_1}+\frac{\theta}{p_2},
$$
 we obtain
\begin{equation}\label{19}
f^{**}(2^{-j})\le c 2^{j/p}B_j^{1-\theta}\beta_{\k(j)}^\theta
+2^{-\nu+1}f^{**}(2^{-j-\nu})+2f^*(2^{-j+\nu}).
\end{equation}
Denote $\s_j= 2^{j/p}B_j^{1-\theta}\beta_{\k(j)}^\theta.$ Recall that
$$
\frac{1}{q}=\frac{1-\theta}{q_1}+\frac{\theta}{q_2}.
$$
Thus, applying H\"older's inequality, we have
$$
\left(\sum_{j\in\ZZ}2^{-jq/p}\s_j^q\right)^{1/q}=\left(\sum_{j\in\ZZ}B_j^{(1-\theta)q}\beta_{\k(j)}^{\theta q}\right)^{1/q}
$$
$$
\le \left(\sum_{j\in\ZZ}B_j^{q_1}\right)^{(1-\theta)/q_1}\left(\sum_{j\in\ZZ}\b_{\k(j)}^{q_2}\right)^{\theta/q_2}.
$$
Using this estimate and inequalities (\ref{3}), (\ref{12}),
(\ref{15}), and (\ref{18}), we obtain
\begin{equation}\label{20}
\left(\sum_{j\in\ZZ}2^{-jq/p}\s_j^q\right)^{1/q}\le c \|\nabla f\|_{p_1,q_1}^{1-\theta}\|f\|_{\dot{B}^s_{p_2,q_2}}^\theta.
\end{equation}

Now we assume that $f\in L^{p,q}$ and we consider the last two terms
on the right hand side of (\ref{19}). We have
$$
2^{-\nu }\left(\sum_{j\in\ZZ}2^{-jq/p}f^{**}(2^{-j-\nu})^q\right)^{1/q}
$$
\begin{equation}\label{21}
= 2^{-\nu/p'}\left(\sum_{j\in\ZZ}2^{-jq/p}f^{**}(2^{-j})^q\right)^{1/q}
\end{equation}
and
$$
\left(\sum_{j\in\ZZ}2^{-jq/p}f^{**}(2^{-j+\nu})^q\right)^{1/q}
$$
\begin{equation}\label{22}
=
2^{-\nu/p}\left(\sum_{j\in\ZZ}2^{-jq/p}f^{**}(2^{-j})^q\right)^{1/q}.
\end{equation}
Since $p_1\not=p_2,$ we have that $1<p<\infty.$ Therefore we can
choose $\nu\in \NN$ such that $2^{-\nu/p'}+2^{-\nu/p}<1/4.$ Then,
applying (\ref{19}) -- (\ref{22}), we obtain
\begin{equation}\label{23}
\left(\sum_{j\in\ZZ}2^{-jq/p}f^{**}(2^{-j})^q\right)^{1/q}\le c \|\nabla f\|_{p_1,q_1}^{1-\theta}\|f\|_{\dot{B}^s_{p_2,q_2}}^\theta.
\end{equation}
This proves our theorem for $r=1$, but with additional assumption
that $f\in L^{p,q}$. It remains to show that this assumption in fact
is true (cf. \cite[p.663]{Led}). For this, we prove the following weak-type inequality
\begin{equation}\label{24}
f^{**}(t)\le c (\nabla f)^{**}(t)^{1-\theta}t^{-\theta/p_2}\|f\|_{\dot{B}^s_{p_2,q_2}}^{\theta}.
\end{equation}

First, by (\ref{f-Pf}),
$$
(f-P_hf)^{**}(t)\le c \sqrt{h}(\nabla f)^{**}(t).
$$
Besides,
$$
(P_hf)^{**}(t)\le t^{-1/p_2}\|P_hf\|_{p_2}.
$$
For any $\mu>0$, find $h_\mu\in [\mu,2\mu]$ such that
$$
\|P_{h_\mu} f\|_{p_2}=\min_{h\in [\mu,2\mu]}\|P_hf\|_{p_2}.
$$
Then
$$
\|f\|_{\dot{B}^s_{p_2,q_2}}^{q_2}\ge \int_\mu^{2\mu} h^{-sq_2/2}\|P_hf\|_{p_2}^{q_2}\,\frac{dh}{h}
$$
$$
\ge c_1\|P_{h_\mu} f\|_{p_2}^{q_2}\mu^{-sq_2/2} \quad (c_1>0).
$$
Using  estimates given above,   we have
$$
f^{**}(t)\le (f-P_{h_\mu}f)^{**}(t)+ (P_{h_\mu}f)^{**}(t)
$$
$$
\le c\left[\mu^{1/2}(\nabla f)^{**}(t)+\mu^{s/2}t^{-1/p_2}\|f\|_{\dot{B}^s_{p_2,q_2}}\right]
$$
for any $\mu>0$. Taking
$$
\mu =\left(\frac{t^{-1/p_2}\|f\|_{\dot{B}^s_{p_2,q_2}}}{(\nabla f)^{**}(t)}\right)^{2\theta},
$$
we obtain (\ref{24}).

Since $\nabla f\in L^{p_1,q_1}(\R^n)$, we have that $t^{1/p_1}(\nabla f)^{**}(t)$ is bounded. Thus, it follows from (\ref{24}) that $t^{1/p}f^{**}(t)$ is also bounded. In consequence, there exists $k_0\in\mathbb{Z}$ such that 
\begin{equation*}
2^{-k_0/p}f^{**}(2^{-k_0})\ge t^{1/p}f^{**}(t)/2\text{ for any }t>0.
\end{equation*} 
Let $\nu>0$. Then, for any integer $K\ge |k_0|$,
\begin{equation*}
2^{-\nu q}\sum_{j=-K}^K 2^{-jq/p}f^{**}(2^{-j-\nu})^q=2^{-\nu q/p'}\sum_{j=-K+\nu}^{K+\nu} 2^{-jq/p}f^{**}(2^{-j})^q\le 
\end{equation*}
\begin{equation*}
\le 2^{-\nu q/p'}\left(\sum_{j=-K}^{K} 2^{-jq/p}f^{**}(2^{-j})^q+\sum_{j=K+1}^{K+\nu} 2^{-jq/p}f^{**}(2^{-j})^q\right)\le 
\end{equation*}
\begin{equation*}
\le 2^{-\nu q/p'}\left(\sum_{j=-K}^{K} 2^{-jq/p}f^{**}(2^{-j})^q+\nu 2^q 2^{-k_0q/p}f^{**}(2^{-k_0})^q\right)\le 
\end{equation*}
\begin{equation*}
\le 2^{-\nu q/p'}\left(1+ 2^{q} \nu\right)    \sum_{j=-K}^{K} 2^{-jq/p}f^{**}(2^{-j})^q.
\end{equation*}

Similarly, we have
\begin{equation*}
\sum_{j=-K}^K 2^{-jq/p}f^{**}(2^{-j+\nu})^q=2^{-\nu q/p}\sum_{j=-K-\nu}^{K-\nu} 2^{-jq/p}f^{**}(2^{-j})^q\le 
\end{equation*}
\begin{equation*}
 \le 2^{-\nu q/p}(1+2^{q}\nu)\sum_{j=-K}^{K} 2^{-jq/p}f^{**}(2^{-j})^q.
\end{equation*}

We apply these estimates (for $\nu$ big enough) to the
inequality
$$
f^{**}(2^{-j})\le c \s_j
+2^{-\nu+1}f^{**}(2^{-j-\nu})+2f^*(2^{-j+\nu})
$$
(see (\ref{19})). Taking into account (\ref{20}), we obtain that for $K\in\NN$, $K\ge |k_0|$
$$
\left(\sum_{j=-K}^K 2^{-jq/p}f^{**}(2^{-j})^q\right)^{1/q}\le c \|\nabla f\|_{p_1,q_1}^{1-\theta}\|f\|_{\dot{B}^s_{p_2,q_2}}^\theta.
$$
This completes the proof of our theorem for $r=1.$

Now we apply induction. Assume that theorem is true for $r-1$ $(r\ge
2).$ Set
$$
\theta'=\frac{1}{r-s} \quad\mbox{and}\quad \bar{\theta}= \frac{r-1}{r-1-s}.
$$
Further, let
\begin{equation}\label{pbar}
\frac{1}{\bar{p_1}}= \frac{1-\theta'}{p_1}+ \frac{\theta'}{p_2}\qquad \frac{1}{\bar{q_1}}= \frac{1-\theta'}{q_1}+ \frac{\theta'}{q_2}.
\end{equation}
Observe that, since $p_1\neq p_2$, then $\bar{p_1}\neq p_2$.
Moreover, using (\ref{100}) and (\ref{pbar}), we obtain
\begin{equation*}
 \frac{1}{p}=\frac{1-\bar{\theta}}{\bar{p_1}}+\frac{\bar{\theta}}{p_2},\quad \frac{1}{q}=\frac{1-\bar{\theta}}{\bar{q_1}}+\frac{\bar{\theta}}{q_2}.
\end{equation*}
Thus, by our inductive assumption
\begin{equation}\label{r-1}
\|f\|_{p,q}\le c \|\mathcal{D}^{r-1}f\|_{\bar{p_1},\bar{q_1}}^{1-\bar{\theta}}\|f\|_{\dot{B}^s_{p_2,q_2}}^{\bar{\theta}}
\end{equation}

We have $\theta'=1/(1-s'),$ where $s'=s+1-r<0.$ Thus, as it was
already proved, for any function $g\in W_{p_1,q_1}^1(\R^n)\cap
\dot{B}^{s+1-r}_{p_2,q_2}(\R^n)$,
\begin{equation*}
\|g\|_{\bar{p_1},\bar{q_1}}\le c \|Dg \|_{p_1,q_1}^{1-\theta'}\|g\|_{\dot{B}_{p_2,q_2}^{s+1-r}}^{\theta'}
\end{equation*}
We apply this inequality to each of the derivatives $D^{\alpha}f$ of
order $|\alpha|=r-1.$  Taking into account that
$$
\|D^{\alpha}f\|_{\dot{B}_{p_2,q_2}^{s+1-r}}\le c
\|f\|_{\dot{B}_{p_2,q_2}^{s}}
$$
(see Lemma \ref{lifting_termico}), we obtain
\begin{equation}\label{r=1}
\|\mathcal{D}^{r-1}f\|_{\bar{p_1},\bar{q_1}}\le c \|\mathcal{D}^rf \|_{p_1,q_1}^{1-\theta'}\|f\|_{\dot{B}_{p_2,q_2}^{s}}^{\theta'}
\end{equation}
We have $(1-\bar{\theta})(1-\theta')=1-\theta$ and
$\bar{\theta}+(1-\bar{\theta})\theta'=\theta$.  Hence, (\ref{r-1})
and (\ref{r=1}) imply (\ref{1}).

\end{proof}

\begin{rem}\label{constant} The explicit value of the constant $c$ in (\ref{1}) is rather complicated. From (\ref{3}), we can see that this constant   remains bounded if
$p_1$ and $q_1$ tend to 1 in such a way that $p_1'/q_1'$ is bounded
(for example, if $1<q_1\le p_1$.) However, if $q_1>1$ is fixed and
$p_1\to 1+,$ then $c\to\infty$.

Also, $c$ blows up if $1/p_1-1/p_2\to 0$ (see (\ref{delta}),
(\ref{12}), and (\ref{15})).

\end{rem}

A special case of Theorem \ref{Sobolev1} is the following theorem.
\begin{teo}\label{Sobolev2} Let $r\in\NN,\,\,1\le r<n, \,\, 1\le p<n/r,$ and let $p^*=np/(n-rp).$ Then for any function $f\in W_p^r(\R^n)$
\begin{equation}\label{sobolev}
\|f\|_{p^*,p}\le c\|\mathcal {D}^rf\|_p^{1-pr/n} \|f\|_{\dot{B}^{r-n/p}_{\infty,p}}^{pr/n}.
\end{equation}
\end{teo}
By virtue of (\ref{embed2}), this result gives a refinement of
Sobolev type inequality (\ref{Sobolev_Lorentz}).

\section{Auxiliary propositions for Triebel-Lizorkin and Besov inequalities}

The following lemma presents a modification of Lemma 2.1 in
\cite{KyP}. It can be interpreted as a continuous counterpart of
Lemma \ref{KP}.
\begin{lem}\label{masmenosdelta}
Let $\gamma>0$. Let $\phi\in L^{q}(\R_+,dt/t)$ be a non-negative function such that
$\phi(t)t^\gamma$ increases or $\phi(t)t^{-\gamma}$ decreases. Then,
for any $\delta>0$, there exists a continuously differentiable
function $\widetilde{\phi}$ on $\R_+$ such that:
\begin{enumerate}
\item[(i)] $\phi(t)\le \widetilde{\phi}(t)$, $t\in\R_+$;
\item[(ii)]  $\widetilde{\phi}(t)t^{\delta}$ increases and $\widetilde{\phi}(t)t^{-\delta}$ decreases on $\R_+$.
\item[(iii)] $\|\widetilde{\phi}\|_{L^q(\R_+,dt/t)}\le c \|\phi\|_{L^q(\R_+,dt/t)}$,
\end{enumerate}
where $c=\left(2(1+\gamma/\delta)\right)^{1/q}$ if $q<\infty$ and
$c=1$ if $q=\infty$.
\end{lem}

\begin{proof}
If $q=\infty$, we define the constant function
$\widetilde{\phi}(t)=\|\phi\|_\infty$, and the lemma follows
immediately. Assume that $q<\infty$. Suppose first that
$\phi(t)t^\gamma$ increases. Set
\begin{equation*}
\phi_1(t)=\left((\delta+\gamma)q\right)^{1/q} t^\delta \left(\int_t^\infty \phi(u)^q u^{-\delta q}\frac{du}{u}\right)^{1/q}.
\end{equation*}
Then $\phi_1(t)t^{-\delta}$ decreases and
\begin{equation*}
\phi_1(t)\ge \left((\delta+\gamma)q\right)^{1/q} t^{\delta+\gamma}\phi(t) \left(\int_t^\infty  u^{-(\delta+\gamma)q}\frac{du}{u}\right)^{1/q}= \phi(t).
\end{equation*}
Furthermore, applying Fubini's theorem, we easily get that
\begin{equation}\label{hardy1}
\|\phi_1\|_{L^q(\R_+,dt/t)}\le\left(1+\frac{\gamma}{\delta}\right)^{1/q}\|\phi\|_{L^q(\R_+,dt/t)}.
\end{equation}
Set now
\begin{equation}\label{varphi}
\widetilde{\phi}(t)=(2\delta q)^{1/q} t^{-\delta}\left(\int_0^t \phi_1(u)^q u^{\delta q}\frac{du}{u}\right)^{1/q}.
\end{equation}
Then $\widetilde{\phi}(t)t^{\delta}$ increases on $\R_+$ and
\begin{equation*}
\widetilde{\phi}(t)\ge \phi_1(t)\ge \phi(t),\quad t\in\R_+.
\end{equation*}
Furthermore, the change of variable $v=u^{2\delta q}$ in the
right-hand side of (\ref{varphi}) gives that
\begin{equation*}
t^{-\delta}\widetilde{\phi}(t)=\left( t^{-2\delta q}\int_0^{t^{2\delta q}} \eta(v^{1/(2\delta q)})
dv\right)^{1/q},
\end{equation*}
where $\eta(u)=(\phi_1(u)u^{-\delta})^q$ is a decreasing function on
$\R_+$. Thus, $t^{-\delta}\widetilde{\phi}(t)$ decreases. Finally,
using Fubini's theorem and $(\ref{hardy1})$, we get (iii).

Let us consider the case when $\phi(t)t^{-\gamma}$ decreases on
$\R_+$. Setting $h(t)=\phi(1/t),$ we have that $h(t)t^\gamma$
increases. As above, we obtain that there exists $\widetilde{h}$
satisfying (i), (ii) and (iii) with respect to $h$. We set
$\widetilde{\phi}(t)=\widetilde{h}(1/t)$. It is easy to see that
$\widetilde{\phi}$ satisfies (i), (ii) and (iii) respect to $\phi$.
The lemma is proved.
\end{proof}

\begin{lem}\label{zetadet}
Let $\alpha, \beta>0$. Let $\varphi$, $\psi:
(0,\infty)\longrightarrow(0,\infty)$ be differentiable functions.
Assume that either $\varphi(t)t^{\alpha}$  decreases or
$\varphi(t)t^{-\alpha}$ increases. Further, assume   that $\psi$ is
bijective, $\psi'(t)>0$ for all $t>0$, and $\psi(t)t^{-\beta}$
decreases. Then the function $z(t)=\psi^{-1}(\varphi(t))$ is
monotone and bijective on $\R_+$ and satisfies inequality
\begin{equation*}
\frac{\alpha}{\beta t}\le \frac{|z'(t)|}{z(t)}\qquad \text{ for any }t>0.
\end{equation*}
\end{lem}
\begin{proof}
If $\varphi(t)t^{\alpha}$ decreases, then $\varphi$ is bijective and
strictly decreasing. The derivative of $\varphi(t)t^{\alpha}$ is
smaller or equal than zero and thus
\begin{equation*}
\frac{\alpha}{t}\le -\frac{\varphi'(t)}{\varphi(t)}.
\end{equation*}
If $\varphi(t)t^{-\alpha}$ increases, then $\varphi$ is bijective
and strictly increasing. In this case, taking the derivative of
$\varphi(t)t^{-\alpha}$ we have
\begin{equation*}
\frac{\alpha}{t}\le \frac{\varphi'(t)}{\varphi(t)}.
\end{equation*}
In any case
\begin{equation}\label{zeta}
\frac{\alpha}{t}\le \frac{|\varphi'(t)|}{\varphi(t)}.
\end{equation}
Now we consider the function $\psi$. Since $\psi(t)t^{-\beta}$
decreases, then $\psi^{-1}(t)t^{-1/\beta}$ increases. Proceeding as
before, we obtain
\begin{equation}\label{zeta2}
\frac{1}{\beta t}\le \frac{(\psi^{-1})'(t)}{\psi^{-1}(t)}.
\end{equation}
Finally, using (\ref{zeta2}) and (\ref{zeta}), we have
\begin{equation*}
\frac{|z'(t)|}{z(t)}=\frac{(\psi^{-1})'(\varphi(t))\,|\varphi'(t)|}{\psi^{-1}(\varphi(t))}\ge\frac{|\varphi'(t)|}{\beta \varphi(t)}\ge \frac{\alpha}{\beta t}.
\end{equation*}
\end{proof}

\begin{lem}\label{equilibrio} Assume that $0< p_1,p_2,q_1,q_2\le \infty$, $p_1\neq p_2$. Let $\rho >0$ and $\s<0$. We
set
\begin{equation}\label{notacion}
\theta=\frac{\rho}{\rho-\s},\qquad\frac{1}{p}=\frac{1-\theta}{p_1}+\frac{\theta}{p_2},\qquad \frac{1}{q}=\frac{1-\theta}{q_1}+\frac{\theta}{q_2}.
\end{equation}
Let $\phi_1\in L^{q_1}(\R_+,dt/t)$ and $\phi_2\in
L^{q_2}(\R_+,dt/t)$ be  non-negative functions. We consider three
following cases, defining a function $\Phi(z,t)$ for $z,t>0$ in each
of them:

(i) let $t^{\rho}\phi_1(t)$ increase and $t^{\s}\phi_2(t)$ decrease
on $\R_+$, and set
$$
\Phi(z,t)= t^{-1/p_1}z^{\rho}\phi_1(z)+t^{-1/p_2}z^{\s}\phi_2(z);
$$

(ii) let $t^{-1/p_1}\phi_1(t)$  and $t^{\s}\phi_2(t)$ decrease on
$\R_+$, and set
$$
\Phi(z,t)= t^{-1/p_1}\phi_1(t)z^{\rho}+t^{-1/p_2}z^{\s}\phi_2(z);
$$

(iii) let $t^{\rho}\phi_1(t)$ increase
 and $t^{-1/p_2}\phi_2(t)$ decrease on $\R_+$,
and set
$$
\Phi(z,t)=t^{-1/p_1}z^{\rho}\phi_1(z)+t^{-1/p_2}\phi_2(t)z^{\s}.
$$
Let $f(t)=\inf_{z>0}\Phi(z,t)\,\, (t>0).$ Then in each of the cases
(i)-(iii)
\begin{equation}\label{lem_equilibrio}
\|f\|_{p,q}\le c\|\phi_1\|_{L^{q_1}(\R_+,dt/t)}^{1-\theta}\|\phi_2\|_{L^{q_2}(\R_+,dt/t)}^{\theta},
\end{equation}
where $c$ is a constant that does not depend on $\phi_1$, $\phi_2$.
\end{lem}
\begin{proof} We first consider  the case (i).
We apply Lemma \ref{masmenosdelta} to $\phi_1$  (with $\g_1=\rho, \,
\d_1=\rho/2),$ and to $\phi_2$ (with $\g_2=|\s|,\, \d_2=|\s|/2).$
 We obtain strictly positive and differentiable
functions $\widetilde{\phi}_1\ge \phi_1$ and $\widetilde{\phi}_2\ge
\phi_2$ such that
\begin{equation}\label{crecedecr1}
\widetilde{\phi}_i(t)t^{\d_i}\quad\text{increase and}\quad \widetilde{\phi}_i(t)t^{-\d_i}\quad\mbox{decrease}
\end{equation}
for $i=1,2.$ Besides
\begin{equation}\label{normaphi}
\|\widetilde{\phi}_i\|_{L^{q_i}(\R_+,dt/t)}\le c\|\phi_i\|_{L^{q_i}(\R_+,dt/t)} \quad (i=1,2).
\end{equation}

Then, for any $t>0$, we have the inequality
\begin{equation}\label{est_suma}
f(t)\le \inf_{z>0}\left[ t^{-1/p_1}z^{\rho}\widetilde{\phi}_1(z)+t^{-1/p_2}
z^{\s}\widetilde{\phi}_2(z)\right].
\end{equation}
Fix $t>0$ and set
$\psi(z)=z^{\rho-\s}\widetilde{\phi_1}(z)/\widetilde{\phi_2}(z).$ By
(\ref{crecedecr1}), the function $\psi(z)z^{-(\rho-\s)/2}$
increases. Thus, $\psi(z)$ is a bijective strictly increasing
function with strictly positive derivative. Furthermore,
(\ref{crecedecr1})  imply also that $\psi(z)z^{-3(\rho-\s)/2}$
decreases. Denote $d=1/p_1-1/p_2.$ We apply Lemma \ref{zetadet} with
$\varphi(t)=t^d, \,\, \a=|d|,$
 and $\beta=3(\rho-\s)/2$. Then,
$z(t)=\psi^{-1}(t^d)$ is a bijective and differentiable function
from $(0,\infty)$ onto $(0,\infty)$ such that
\begin{equation}\label{zetaprimadet}
\frac{1}{t}\le c\frac{|z'(t)|}{z(t)}.
\end{equation}
Choosing $z\equiv z(t)$ in (\ref{est_suma}), the two addends are
equal. In consequence,
\begin{equation*}
f(t)\le 2\left(t^{-1/p_1}z(t)^{\rho}\widetilde{\phi}_1(z(t))\right)^{1-\theta}
\left(t^{-1/p_2}
z(t)^{\s}\widetilde{\phi}_2(z(t))\right)^\theta
\end{equation*}
\begin{equation*}
=2t^{-1/p} \left(\widetilde{\phi}_1(z(t))\right)^{1-\theta}
\left(\widetilde{\phi}_2(z(t))\right)^\theta.
\end{equation*}
Now we apply H\"older's inequality with conjugate exponents
$q_1/(q(1-\theta))$ and $q_2/(q\theta)$.  Using also
  (\ref{zetaprimadet}), we obtain
\begin{equation*}
\|f\|_{p,q}\le 2\left(\int_0^\infty\left(\widetilde{\phi}_1(z(t))\right)^{q(1-\theta)}
\left(\widetilde{\phi}_2(z(t))\right)^{q\theta}\frac{dt}{t}\right)^{1/q}
\end{equation*}
\begin{equation*}
\le 2\left(\int_0^\infty\widetilde{\phi}_1(z(t))^{q_1}\frac{dt}{t}\right)^{\frac{1-\theta}{q_1}}
\left(\int_0^\infty\widetilde{\phi}_2(z(t))^{q_2}\frac{dt}{t}\right)^{\frac{\theta}{q_2}}
\end{equation*}
\begin{equation*}
\le c
\|\widetilde{\phi}_1\|_{L^{q_1}(\R_+,dt/t)}
^{1-\theta}
\|\widetilde{\phi}_2\|_{L^{q_2}(\R_+,dt/t)}^\theta.
\end{equation*}
This inequality and (\ref{normaphi}) imply (\ref{lem_equilibrio}) in
the case (i).

Next, we consider the case (ii). We apply Lemma \ref{masmenosdelta}
to $\phi_2$ as above and to $\phi_1$ with $\g_1=1/p_1$ and
$\d_1=|d|/2$ ($d=1/p_1-1/p_2)$. We obtain strictly positive and
differentiable functions $\widetilde{\phi}_1\ge \phi_1$ and
$\widetilde{\phi}_2\ge \phi_2$ satisfying (\ref{crecedecr1}) and
(\ref{normaphi}). Then, we set
$$
\varphi(t)=\frac{t^d}{\widetilde{\phi}_1(t)}\quad\mbox{and}\quad
\psi(z)=\frac{z^{\rho-\s}}{\widetilde{\phi_2}(z)}.
$$
By (\ref{crecedecr1}),  $\varphi(t)t^{|d|/2}$ decreases if $d<0$ and
$\varphi(t)t^{-|d|/2}$ increases if $d>0.$ As above,
$\psi(z)z^{-(\rho-\s)/2}$ increases and $\psi(z)z^{-3(\rho-\s)/2}$
decreases. Thus, we can apply Lemma \ref{zetadet} with $\a=|d|/2$
 and $\beta=3(\rho-\s)/2$. Then,
$z(t)=\psi^{-1}(\varphi(t))$ is a bijective and differentiable
function from $(0,\infty)$ onto $(0,\infty)$ satisfying
(\ref{zetaprimadet}). We have that for any $t>0$ and any $z>0$
$$
f(t)\le t^{-1/p_1}\widetilde{\phi}_1(t)z^{\rho}+t^{-1/p_2}z^{\s}\widetilde{\phi}_2(z).
$$
Choosing $z=z(t),$ we get
$$
f(t)\le 2t^{-1/p} \left(\widetilde{\phi}_1(t)\right)^{1-\theta}
\left(\widetilde{\phi}_2(z(t))\right)^\theta.
$$
Proceeding as above, we obtain inequality (\ref{lem_equilibrio}) in
the case (ii).

Finally, the case (iii) is treated by similar arguments. Moreover,
it can also be derived from the case (ii) by exchanging $p_1$ for
$p_2$, $q_1$ for $q_2$, $\rho$ for $|\sigma|$ and $z$ for $1/z$.
\end{proof}

\begin{lem}\label{lem_termico}
Let $f\in S_0(\mathbb{R}^n)$. Assume also that $f\in
L^1(\mathbb{R}^n)+L^\infty(\mathbb{R}^n)$. Let $m\in\mathbb{N}$.
Then, for almost all $x\in\mathbb{R}^n$,
\begin{equation}\label{eq_termica}
f(x)=\frac{(-1)^m}{(m-1)!}\int_0^\infty h^{m-1}\frac{\partial^m}{\partial h^m}P_hf(x) dh.
\end{equation}
\end{lem}
\begin{proof}
First, let us compute the derivatives of the heat kernel
$p_h(y)=(4\pi h)^{-n/2}e^{-|y|^2/(4h)}$. It can be seen that
\begin{equation}\label{laguerre}
\frac{\partial^m}{\partial h^m}p_h(y)=p_h(y)\frac{m!(-1)^m}{h^m}L_m^{n/2-1}\left(\frac{|y|^2}{4h}\right),
\end{equation}
where $L_m^{n/2-1}$ is the generalized Laguerre polynomial. To prove
(\ref{laguerre}), for instance, use \cite[p.190(26)]{EH} if $y\neq
0$ and \cite[p.189(13)]{EH} if $y=0$. Besides, by
\cite[p.190(28)]{EH}, it follows that
\begin{equation}\label{derivada_integrando}
\frac{\partial}{\partial h}\left(p_h(y)L_{m-1}^{n/2}\left(\frac{|y|^2}{4h}\right)\right)=-m\frac{p_h(y)}{h}L_m^{n/2-1}\left(\frac{|y|^2}{4h}\right).
\end{equation}

Now we will show that for any $h>0$
\begin{equation}\label{passderivative}
\frac{\partial^m}{\partial h^m}P_hf(x)=\int_{\R^n}\frac{\partial^m}{\partial h^m}p_h(y)f(x-y)dy.
\end{equation}
First we can assume that $1/M<h<M$. Then, by substituting in
$(\ref{laguerre})$ $h$ for $M$ or $1/M$ when convenient, we can
bound $|\frac{\partial^m}{\partial h^m}p_h(y)|$ by a function
independent of $h$ which is in the Schwartz class. And the same can
be done with the derivatives of order $i$, $i=0,1,\ldots,m$. Thus,
we can pass the derivative through the integral sign (see, for
instance, \cite[Corollary 5.9]{Bar}) and (\ref{passderivative}) is
true.

We will compute explicitly the integral in (\ref{eq_termica}).
Define for $m\in\mathbb{N}$,
\begin{equation}\label{efesubeme}
F_m(h,x)=-\int_{\mathbb{R}^n} p_h(y)L_{m-1}^{n/2}\left(\frac{|y|^2}{4h}\right)f(x-y)dy,
\end{equation}
Using the same arguments as before, the partial derivative of $F_m$
with respect to $h$ can be calculated passing through the integral
sign. Thus, (\ref{efesubeme}),  (\ref{derivada_integrando}),
(\ref{laguerre}), and (\ref{passderivative}) lead to
\begin{equation*}
\frac{\partial F_m}{\partial h}(h,x)=\frac{(-1)^m}{(m-1)!} h^{m-1}\frac{\partial^m}{\partial h^m}P_hf(x).
\end{equation*}
In other words, $F_m$ is a primitive for the right hand side of the
last equation. It only remains to prove that for any  $m\in
\mathbb{N}$,
\begin{equation}\label{LIMITS1}
\lim_{h\to +\infty}F_m(h,x)=0 \quad\mbox{for almost all}\quad x\in\mathbb{R}^n
\end{equation}
and
\begin{equation}\label{LIMITS2}
\lim_{h\to 0^+}F_m(h,x)=-f(x) \quad\mbox{for almost all}\quad x\in\mathbb{R}^n,
\end{equation}
 and (\ref{eq_termica}) follows.

Now, since $L_{m-1}^{n/2}=\sum_{i=0}^{m-1}L_i^{n/2-1}$ (cf.
\cite[p.192(38)]{EH}), and Laguerre polynomials $L_{i}^{n/2-1}$ are
orthogonal with respect to the weight $e^{-t}t^{n/2-1}$,
\begin{equation*}
\int_0^\infty e^{-u}L_{m-1}^{n/2}(u)u^{n/2-1}du=\sum_{i=0}^{m-1}\int_0^\infty e^{-u}L_{i}^{n/2-1}(u)u^{n/2-1}du
\end{equation*}
\begin{equation*}
=\int_0^\infty e^{-u}L_{0}^{n/2-1}(u)u^{n/2-1}du=\int_0^\infty e^{-u}u^{n/2-1}du=\Gamma(n/2).
\end{equation*}
Furthermore, changing to spherical coordinates, applying the change
of variable $t^2=u$, and the last equality, we obtain
\begin{equation*}
\frac{1}{\pi^{n/2}}\int_{\mathbb{R}^n}e^{-|z|^2}L_{m-1}^{n/2}(|z|^2)dz=\frac{|\mathbb{S}^{n-1}|}{\pi^{n/2}}\int_0^\infty e^{-t^2}L_{m-1}^{n/2}(t^2)t^{n-1}dt
\end{equation*}
\begin{equation*}
=\frac{|\mathbb{S}^{n-1}|}{2\pi^{n/2}}\int_0^\infty e^{-u}L_{m-1}^{n/2}(u)u^{n/2-1}du=\frac{|\mathbb{S}^{n-1}|}{2\pi^{n/2}}\Gamma(n/2)=1.
\end{equation*}
In conclusion, $\pi^{-n/2}e^{-|z|^2}L_{m-1}^{n/2}(|z|^2)$ is an
integrable function with the integral equal to $1$. Then
$g_h(y)=p_h(y)L_{m-1}^{n/2}(|y|^2/(4h))$ can be used to construct by
convolution an approximation of the identity when $h\to 0^+$.
Finally, given any $\varepsilon>0$, as $f\in S_0(\mathbb{R}^n)$, it
holds that $|\{f>\varepsilon\}|<\infty$. Thus, we can split
$f=f_1^\varepsilon+f_2^\varepsilon$, where
$f_1^\varepsilon=f\chi_{\{f>\varepsilon\}}$ and
$f_2^\varepsilon=f\chi_{\{f\le\varepsilon\}}$. Then
$\|f_2^\varepsilon\|_\infty\le \varepsilon$. Furthermore, since
$f\in L^1(\mathbb{R}^n)+L^\infty(\mathbb{R}^n)$ and
$|\{f>\varepsilon\}|<\infty$, we have that $f_1^\varepsilon\in
L^1(\mathbb{R}^n).$ It is clear that the function
$$
\psi(x)=\sup_{|z|\ge|x|} e^{-|z|^2}|L_{m-1}^{n/2}(|z|^2)|\quad (x\in \R^n)
$$
is
integrable on $\R^n$ and therefore
\begin{equation*}
\lim_{h\to 0^+}g_h*f_1^\varepsilon(x)=f_1^\varepsilon(x)\quad\mbox{for almost all}\quad x\in\mathbb{R}^n
\end{equation*}
(see \cite[Ch. 1, Theorem 1.25]{SW}).
 Thus, by (\ref{efesubeme}), for  almost all $x\in\mathbb{R}^n$,
\begin{equation*}
\limsup_{h\to 0^+}|F_m(h,x)+f(x)|= \limsup_{h\to 0^+}|-g_h*f_1^\varepsilon(x)-g_h*f_2^\varepsilon(x)+f_1^\varepsilon(x)+f_2^\varepsilon(x)|
\end{equation*}
\begin{equation*}
\le\lim_{h\to 0^+}|g_h*f_1^\varepsilon(x)-f_1^\varepsilon(x)|+ \limsup_{h\to 0^+}|g_h*f_2^\varepsilon(x)-f_2^\varepsilon(x)|
\end{equation*}
\begin{equation*}
 \le \limsup_{h\to 0^+}\|f_2^\varepsilon\|_\infty (1+\|g_h\|_1)\le (1+\|g_1\|_1)\varepsilon.
\end{equation*}
This implies (\ref{LIMITS2}). On the other hand
\begin{equation*}
\limsup_{h\to +\infty}\| F_m(h,\cdot)\|_\infty \le \limsup_{h\to +\infty}\| g_h*f_1^\varepsilon\|_\infty+\| g_h*f_2^\varepsilon\|_\infty
\end{equation*}
\begin{equation*}
\le\limsup_{h\to +\infty}\| g_h\|_\infty \|f_1^\varepsilon\|_1+\| g_h\|_1\|f_2^\varepsilon\|_\infty= 0+ \|g_1\|_1\|f_2^\varepsilon\|_\infty\le \|g_1\|_1\cdot\varepsilon.
\end{equation*}
This yields (\ref{LIMITS1}). The proof is completed.
\end{proof}

\section{Inequalities with Triebel-Lizorkin and Besov norms}

\begin{teo}\label{teoFF}
Let $0<p_1, p_2<\infty$, $0<q_1,q_2\le \infty$, $r>0$, $s<0$. Let
\begin{equation}\label{new_notation}
\theta=\frac{r}{r-s},\qquad\frac{1}{p}=\frac{1-\theta}{p_1}+\frac{\theta}{p_2},\qquad \frac{1}{q}=\frac{1-\theta}{q_1}+\frac{\theta}{q_2}.
\end{equation}
 Then, for any function $f\in
S_0(\mathbb{R}^n)\cap\left(L^1(\mathbb{R}^n)+L^\infty(\mathbb{R}^n)\right)$
it holds that
\begin{equation}\label{est_teoFF}
\|f\|_{p,q}\le c\|f\|_{\dot{F}^{r}_{p_1,q_1;\infty}}^{1-\theta} \|f\|_{\dot{F}^{s}_{p_2,q_2;\infty}}^{\theta}
\end{equation}
where 
$c$ does not depend on $f$.
\end{teo}

\begin{proof}
First we assume that the quasinorms in the right hand side of
(\ref{est_teoFF}) are finite. Otherwise, the result is trivial.

Choose a natural $m$ such that $m>r/2.$ By Lemma \ref{lem_termico}
we have
\begin{equation}\label{repres_int_f}
f(x)=\frac{(-1)^m}{(m-1)!}\int_0^\infty h^{m-1}\frac{\partial^m}{\partial h^m}P_hf(x)dh.
\end{equation}
Then, for any $z>0$ we obtain
$$
|f(x)|\le 
\frac{1}{(m-1)!}\int_0^z h^{m-1}\left|\frac{\partial^m}{\partial h^m}P_hf(x)\right|dh
$$
\begin{equation}\label{est}
+\frac{1}{(m-1)!} \int_z^\infty h^{m-1}\left|\frac{\partial^m}{\partial
h^m}P_hf(x)\right|dh.
\end{equation}
Now, we set
\begin{equation}\label{HyG}
H(x)=\sup_{h>0}h^{m-r/2}\left|\frac{\partial^m}{\partial h^m}P_hf(x)\right|,\,\,
G(x)=\sup_{h>0}h^{m-s/2}\left|\frac{\partial^m}{\partial h^m}P_hf(x)\right|.
\end{equation}
Note that $\|H\|_{p_1,q_1}= \|f\|_{\dot{F}^{r}_{p_1,q_1;\infty}}$
and $\|G\|_{p_2,q_2}= \|f\|_{\dot{F}^{s}_{p_2,q_2;\infty}}$.

Besides, by (\ref{est})
\begin{equation*}
|f(x)|\le 
\frac{1}{(m-1)!}\left(H(x)\int_0^z h^{r/2}\frac{dh}{h}+ G(x)\int_z^\infty h^{s/2}\frac{dh}{h}\right),
\end{equation*}
and, taking non-increasing rearrangements, we get
\begin{equation}\label{estHG}
f^*(2t)\le  
\frac{2}{(m-1)!}
\left(\frac{ z^{r/2}}{r}H^*(t)+ \frac{z^{s/2}}{|s|}G^*(t)\right)
\end{equation}
Now fix $t>0$. If $H^*(t)\neq 0$, we choose in (\ref{estHG})
$$z\equiv z(t)=\left(\frac{G^*(t)r}{H^*(t)|s|}\right)^{2/(r+|s|)}$$.
Then,
\begin{equation}\label{estprod}
f^*(2t)\le  
\frac{4}{(m-1)!r^{1-\theta}|s|^\theta}
H^*(t)^{1-\theta}G^*(t)^\theta.
\end{equation}
Note that if $H^*(t)= 0$, (\ref{estHG}) implies that $f^*(2t)=0$ and
(\ref{estprod}) is also true. This inequality and H\"older's
inequality lead to (\ref{est_teoFF}).
\end{proof}

\begin{rem}\label{mejora_F}
Note that $\|g\|_{p,q}=\|f\|_{\dot{F}^{0}_{p,q;l}}$,
 where $g$ denotes the function
\begin{equation}\label{funcion_g}
g(x)=\left(\int_0^\infty h^{ml}\left|\frac{\partial^m}{\partial h^m}P_hf(x)\right|^l \frac{dh}{h} \right)^{1/l},\quad 0<l<\infty
\end{equation}
Then, using this expression instead of (\ref{repres_int_f}) and following the same reasonings we have that
\begin{equation*}
\|f\|_{\dot{F}^{0}_{p,q;l}}\le 2^{1/p}\left(\frac{4}{r^{1-\theta}|s|^\theta l}\right)^{1/l}\|f\|_{\dot{F}^{r}_{p_1,q_1;\infty}}^{1-\theta} \|f\|_{\dot{F}^{s}_{p_2,q_2;\infty}}^{\theta}.
\end{equation*}
This kind of Gagliardo-Nirenberg inequality was essentially proved
by Oru \cite[p. 395]{ByM}. His approach used the representation of
the  Triebel-Lizorkin norms in terms of Littlewood-Paley
decompositions.
\end{rem}

\begin{teo}\label{teoBB}
Let $1\le p_1,p_2\le \infty$. $1\le q_1,q_2\le \infty$, $r>0$,
$s<0$. Assume the previous notation in (\ref{new_notation}) together
with $p_1\neq p_2$. Then, for any function $f\in
S_0(\mathbb{R}^n)\cap\left(L^1(\mathbb{R}^n)+L^\infty(\mathbb{R}^n)\right)$,
it holds that
\begin{equation}
\label{teobesovtermico}
\|f\|_{p,q}\le c \|f\|_{\dot{B}^{r}_{p_1,q_1}}^{1-\theta}\|f\|_{\dot{B}^{s}_{p_2,q_2}}^\theta,
\end{equation}
where
$c$ does not depend on $f$.
\end{teo}

\begin{proof}
As in the previous theorem, we assume that the quasinorms in the
right hand side of (\ref{teobesovtermico}) are finite, we choose a
natural $m$ such that $m>r/2$, and apply Lemma
\ref{lem_termico}. Then, for any $z>0$, we obtain estimate
(\ref{est}). Thus, taking rearrangements we get
\begin{equation*}
(m-1)!f^*(2t)\le
\end{equation*}
\begin{equation}\label{estreor}
\left(\int_0^z h^{m-1}\frac{\partial^m}{\partial h^m}P_hf(\cdot) dh\right)^*\left(t\right)+
\left(\int_z^\infty h^{m-1}\frac{\partial^m}{\partial h^m}P_hf(\cdot) dh\right)^*\left(t\right).
\end{equation}
Now, applying weak inequalities and Minkowski integral inequality,
we get
\begin{equation*}
(m-1)!f^*(2t)\le
\end{equation*}
\begin{equation*}
 t^{-1/p_1}\left\|\int_0^z h^{m-1}\frac{\partial^m}{\partial h^m}P_hf(\cdot) dh\right\|_{p_1}+ t^{-1/p_2}\left\|\int_z^\infty h^{m-1}\frac{\partial^m}{\partial h^m}P_hf(\cdot) dh\right\|_{p_2}
\end{equation*}
\begin{equation*}
\le t^{-1/p_1}\int_0^z h^{m}\left\|\frac{\partial^m}{\partial h^m}P_hf\right\|_{p_1} \frac{dh}{h}+t^{-1/p_2}\int_z^\infty h^{m}\left\|\frac{\partial^m}{\partial h^m}P_hf\right\|_{p_2}\frac{dh}{h}
\end{equation*}
\begin{equation}\label{estBB}
=t^{-1/p_1}z^{r/2}\phi_1(z)+t^{-1/p_2}z^{-|s|/2}\phi_2(z),
\end{equation}
where
\begin{equation}\label{def_of_phi_1(z)}
\phi_1(z)=z^{-r/2}\int_0^z h^{m}\left\|\frac{\partial^m}{\partial h^m}P_hf\right\|_{p_1} \frac{dh}{h}
\end{equation}
and
\begin{equation}\label{def_of_phi_2(z)}
\phi_2(z)=z^{|s|/2} \int_z^\infty h^{m}\left\|\frac{\partial^m}{\partial h^m}P_hf\right\|_{p_2}\frac{dh}{h}.
\end{equation}
Note that $\phi_1(z)z^{r/2}$ is increasing and $\phi_2(z)z^{-|s|/2}$
is decreasing on $\R_+$. By Hardy's inequality \cite[p.124]{BSh} it
holds that
\begin{equation}\label{J}
\left(\int_0^\infty \phi_1(z)^{q_1}\frac{dz}{z}\right)^{1/q_1}\le \frac{2}{r}\|f\|_{\dot{B}^{r}_{p_1,q_1}},\quad 1\le q_1<\infty
\end{equation}
and
\begin{equation}\label{I}
\left(\int_0^\infty \phi_2(z)^{q_2}\frac{dz}{z}\right)^{1/q_2}\le \frac{2}{|s|}\|f\|_{\dot{B}^{s}_{p_2,q_2}},\quad 1\le q_2<\infty.
\end{equation}
Note that the last two inequalities are still valid, with the usual
mo\-di\-fi\-ca\-tions, if $q_1=\infty$ or $q_2=\infty$. We can apply
Lemma \ref{equilibrio} (\textit{i}) to (\ref{estBB}). We use also
(\ref{J}) and (\ref{I}) to get (\ref{teobesovtermico}).
\end{proof}

\begin{rem}
Note that, since in the proof we use weak inequalities, the
$L^{p_1}$, $L^{p_2}-$norms taken in the  Besov norms in
(\ref{teobesovtermico}) can be  replaced by the smaller
Marcinkiewicz norms $L(p_1,\infty)$, $L(p_2,\infty)$.
\end{rem}

\begin{rem}
Let $r>0$, $1\le q\le \infty$, $1\le p<n/r$. Let $f\in
B_{p,q}^r(\mathbb{R}^n)$ and $p^*=np/(n-r p)$. As it was mentioned
in Introduction,  Theorem \ref{teoBB} implies inequality
\begin{equation*}
\|f\|_{p^*,q}\le c\|f\|_{\dot{B}^r_{p,q}}^{1-r p/n}\|f\|_{\dot{B}^{r-n/p}_{\infty,q}}^{r p/n}
\end{equation*}
(proved in \cite{BC} for $p=q$). By (\ref{embed1}), this gives a
refinement of the inequality (\ref{Herz-Peetre}).
\end{rem}

\begin{rem}
Applying H\"older inequality, it is easy to see
$$\|f\|_{\dot{B}^0_{p,q}}\le\|f\|_{\dot{B}^{r}_{p_1,q_1}}^{1-\theta}\|f\|_{\dot{B}^{s}_{p_2,q_2}}^\theta.$$
If $q=p>2$, this inequality is weaker than (\ref{teobesovtermico}),
because it can be proved that in this case $L^{p}\subset
\dot{B}^0_{p,p}$ with proper inclusion \cite[p. 47, Proposition 2,
iii) and p.242, Theorem 1, ii)]{Tri}.
\end{rem}

\begin{rem}
Theorem \ref{teoBB} also admits a counterpart of Remark
\ref{mejora_F}. Thus, if $0<l\le 1$ we can follow the steps of the
proof of Theorem \ref{teoBB} with the function $g$ defined at
(\ref{funcion_g}), obtaining
\begin{equation*}
\|f\|_{\dot{F}^0_{p,q;l}}\le c \|f\|_{\dot{B}^{r}_{p_1,q_1}}^{1-\theta}\|f\|_{\dot{B}^{s}_{p_2,q_2}}^\theta.
\end{equation*}
\end{rem}

\begin{teo}\label{tFB}
Let $0<p_1<\infty$ and $1\le p_2\le \infty$. Let $0< q_1\le
\infty$, $1\le q_2\le \infty$, $r>0$, $s<0$. Assume also the
notation in (\ref{new_notation}) and $p_1\neq p_2$. Then, for any $f\in
S_0(\mathbb{R}^n)\cap\left(L^1(\mathbb{R}^n)+L^\infty(\mathbb{R}^n)\right)$,
it holds that
\begin{equation}
\label{teoFB}
\|f\|_{p,q}\le c \|f\|_{\dot{F}^{r}_{p_1,q_1;\infty}}^{1-\theta}\|f\|_{\dot{B}^{s}_{p_2,q_2}}^\theta,
\end{equation}
where
$c$ does not depend on $f$.
\end{teo}
\begin{proof}
The proof is a mix of the proofs of Theorems \ref{teoFF} and
\ref{teoBB}. We obtain again (\ref{estreor}). The first addend is
estimated like in (\ref{estHG}) and the second one like
(\ref{estBB}). Then we obtain for any $z>0$
\begin{equation*}
(m-1)!f^*(2t)\le t^{-1/p_1}z^{r/2}\phi_1(t)+t^{-1/p_2}z^{-|s|/2}\phi_2(z).
\end{equation*}
Remember that $\phi_2(z)$, defined in (\ref{def_of_phi_2(z)}),
satisfies (\ref{I}) and $\phi_2(z)z^{-|s|/2}$ decreases in $z$.
Besides $\phi_1(t)=2t^{1/p_1}H^*(t)/r\in L^{q_1}(\R_+,dt/t)$, where
$H$ is defined at (\ref{HyG}) and thus,
$\|\phi_1\|_{L^{q_1}(\R_+,dt/t)}=\frac{2}{r}\|f\|_{\dot{F}^{r}_{p_1,q_1;\infty}}$.
It only remains to apply Lemma \ref{equilibrio} (\textit{ii}) and
(\ref{teoFB}) follows.
\end{proof}

\begin{rem}

It is well known that if $1<p<\infty$, then the Triebel-Lizorkin
norms are equivalent to the Sobolev norms. That is:
$\|\cdot\|_{p}\sim \|\cdot\|_{\dot{F}^0_{p;2}}$,
$\|\cdot\|_{\dot{W}_p^r}\sim \|\cdot\|_{\dot{F}^r_{p;2}}$ (see
\cite[p.242, Theorem 1]{Tri}). Therefore, our main result (Theorem
\ref{Sobolev1}) could be seen as a consequence of Theorem \ref{tFB}
in some particular cases. However, it is impossible in the important
case $p=1.$ Besides, we  consider the more general Lorentz
quasinorms rather that $L^p$ norms. This is motivated since in some
contexts in the study of Sobolev inequalities the $L^p$ norms of the
derivatives have revealed not to be a enough precise scale.
Frequently they are substituted for the more precise Lorentz scale
$L^{p,q}$.

The Triebel-Lizorkin-Lorentz spaces are described in terms of
Littlewood-Paley decompositions. It can be proved also \cite[Theorem
5]{YCP} that $\|\cdot\|_{p,q}\sim \|\cdot\|_{\dot{F}^0_{p,q;2}}$
($1<p<\infty$, $0\le q\le \infty$) and using a Bernstein type
inequality $\|\cdot\|_{\dot{W}_{p,q}^r}\sim
\|\cdot\|_{\dot{F}^r_{p,q;2}}$ (similar to \cite[Lemma 2.3]{ZWT}).
But we have to take into account that in our Theorem \ref{tFB} we
are using a thermic definition of the Triebel-Lizorkin-Lorentz
spaces. Although it is reasonable to figure out that the thermic and
Littlewood-Paley definitions are equivalent (following the methods
in \cite{Tri2}), this equivalence is not proved in the literature
and is out of the objectives of this paper.

On the other hand, note that the constant in
$\|\cdot\|_{\dot{F}^{r}_{p_1,p_1;2}}\le c
\|\cdot\|_{\dot{W}^{r}_{p_1,p_1}}$ explodes when $p_1\to 1$, and
therefore Theorem \ref{Sobolev1} can not be derived as a consequence
of Theorem \ref{tFB}.
\end{rem}

\begin{teo}\label{last}
Let $1\le p_1\le \infty$ and $0<p_2< \infty$. Let $1\le q_1\le
\infty$, $0< q_2\le \infty$, $r>0$, $s<0$. Assume also the notation
in (\ref{new_notation}) and $p_1\neq p_2$. Then, for any $f\in
S_0(\mathbb{R}^n)\cap\left(L^1(\mathbb{R}^n)+L^\infty(\mathbb{R}^n)\right)$,
it holds that
\begin{equation}
\label{teoBF}
\|f\|_{p,q}\le c \|f\|_{\dot{B}^{r}_{p_1,q_1}}^{1-\theta}\|f\|_{\dot{F}^{s}_{p_2,q_2;\infty}}^\theta,
\end{equation}
where
$c$ does not depend on $f$.
\end{teo}
\begin{proof}
For proving this theorem we estimate the first addend in
(\ref{estreor}) as in Theorem \ref{teoBB} and the second as in
Theorem \ref{teoFF}. That is,
\begin{equation*}
(m-1)!f^*(2t)\le t^{-1/p_1}z^{r/2}\phi_1(z)+t^{-1/p_2}z^{-|s|/2}\phi_2(t),
\end{equation*}
where $\phi_1(z)$ is defined at (\ref{def_of_phi_1(z)}),
$\phi_1(z)z^{r/2}$ increases and (\ref{J}) holds.
$\phi_2(t)=2t^{1/p_2}G^*(t)/|s|\in L^{q_2}(\R_+,dt/t)$, where $G$ is
defined in (\ref{HyG}). (\ref{teoBF}) follows from Lemma
\ref{equilibrio} (\textit{iii}).
\end{proof}

\begin{rem}
Let us note that the ``constants'' $c$ appearing in
(\ref{est_teoFF}), (\ref{teobesovtermico}), (\ref{teoFB}) and (\ref{teoBF})
depend on the integer number $m$ chosen in the definition of the
Triebel-Lizorkin and Besov quasinorms. The constants can be computed
explicitly, but the expressions are not friendly. Here we only
remark that they explode when $r$ or $s$ tend to zero. Theorems
\ref{teoBB}, \ref{tFB} and \ref{last} use Lemma \ref{equilibrio},
hence the constants also explode when $1/p_1-1/p_2\to 0$.
\end{rem}

\begin{rem}\label{remark_wadade} As it was mentioned in the Introduction, as a consequence of theorems \ref{last} and \ref{teoFF} can be obtained limiting cases of Gagliardo-Nirenberg inequalities similar to those in \cite{Wa}.
To be more concrete, let $1<p<q<\infty$, $0<r,\rho <\infty$. Choose
$\max\{1,q-p,r,\rho\}< p_1<\infty$. From Theorem \ref{last} we have
\begin{equation*}
\|f\|_q\le c \|f\|_{\dot{B}^{n/p_1}_{p_1,p_1}}^{1-p/q}\|f\|_{\dot{F}^{n(1-q/p)/p_1}_{pp_1/(p_1+p-q),\infty}}^{p/q}.
\end{equation*}
Now, by (\ref{embed1}) and (\ref{second_index}),
\begin{equation*}
\|f\|_{\dot{B}^{n/p_1}_{p_1,p_1}}\le c \|f\|_{\dot{B}^{n/r}_{r,p_1}}\le c \|f\|_{\dot{B}^{n/r}_{r,\rho}}.
\end{equation*}
Besides, using well known embeddings (cf. \cite[2.7.1, p.47
Proposition 2. i), p.242 Theorem 1]{Tri})
\begin{equation}\label{embedding_last_remark}
\|f\|_{\dot{F}^{n(1-q/p)/p_1}_{pp_1/(p_1+p-q),\infty}}\le c \|f\|_{\dot{F}^{0}_{p,\infty}}\le c \|f\|_{\dot{F}^{0}_{p,2}}\le c\|f\|_p.
\end{equation}
Putting together the three last inequalities, (\ref{wadade1})
immediately follows.

To prove (\ref{wadade2}) we proceed in the same way, now using
Theorem \ref{teoFF}:
\begin{equation*}
\|f\|_q\le c \|f\|_{\dot{F}^{n/p_1}_{p_1,\infty}}^{1-p/q}\|f\|_{\dot{F}^{n(1-q/p)/p_1}_{pp_1/(p_1+p-q),\infty}}^{p/q}.
\end{equation*}
It also holds that (see \cite[2.7.1]{Tri})
\begin{equation*}
\|f\|_{\dot{F}^{n/p_1}_{p_1,\infty}}\le c \|f\|_{\dot{F}^{n/r}_{r,\infty}}
\end{equation*}
Finally, the two last inequalities and again
(\ref{embedding_last_remark}) imply (\ref{wadade2}).

Let us note that in this remark we are using equivalences of the quasinorms defined in terms of the Gauss-Weierstrass semigroup and of those defined in terms of Littlewood-Paley decompositions. It is known that these equivalences hold modulo polynomials (cf. \cite{Tri2}). However, it is not necessary to consider (\ref{wadade1}) and (\ref{wadade2}) modulo polynomials, since we assume that $f\in L^p(\mathbb{R}^n)$.
\end{rem}

\end{document}